\documentclass[12pt]{amsart}
\usepackage[top=1.5in, left=1.5in, right=1.5in, bottom=1.5in]{geometry}
\usepackage{amsmath}
\usepackage{amsfonts}
\usepackage{amsthm}
\usepackage{amssymb}
\usepackage{graphicx}
\usepackage{tikz}
\usepackage{hyperref}
\usepackage{enumitem}

\usepackage{bm}
\bmdefine{\br}{r}
\bmdefine{\bu}{u}
\newcommand{\buzero}{\bu_0}
\newcommand{\but}{\bu_t}
\bmdefine{\btildeX}{\tilde{X}}
\bmdefine{\btildeY}{\tilde{Y}}
\bmdefine{\btildeB}{\tilde{B}}
\bmdefine{\bX}{X}
\bmdefine{\bY}{Y}
\bmdefine{\bB}{B}

\bmdefine{\bv}{v}
\bmdefine{\bz}{z}
\bmdefine{\bw}{w}
\bmdefine{\bq}{q}
\bmdefine{\bc}{c}
\bmdefine{\bcprime}{c'}
\bmdefine{\brdot}{\dot{r}}
\bmdefine{\blambda}{\lambda}

\newtheorem{introTheorem}{Theorem}
\newtheorem{introProposition}[introTheorem]{Proposition}

\newtheorem{theorem}{Theorem}[section]
\newtheorem{lemma}[theorem]{Lemma}
\newtheorem{proposition}[theorem]{Proposition}
\newtheorem{corollary}[theorem]{Corollary}

\theoremstyle{definition}
\newtheorem{remark}[theorem]{Remark}
\numberwithin{equation}{section}

\newcommand*{\R}{\mathbb{R}}
\newcommand*{\integerinterval}[2]{\{ #1, \dots, #2 \}}
\newcommand*{\sphere}[1]{S^{#1}}
\newcommand*{\torus}[1]{\mathbb{T}^{#1}}
\newcommand*{\vectprod}{\times}

\newcommand*{\isomorphic}{\simeq}
\newcommand*{\set}[1]{\left\lbrace #1 \right\rbrace}
\newcommand*{\abs}[1]{|#1|}
\newcommand*{\identity}[1]{\mathrm{id}_{#1}}
\newcommand*{\diff}[1]{\mathrm{d}{#1}}
\newcommand*{\scalarproduct}[2]{\langle #1 , #2 \rangle}
\newcommand*{\norm}[1]{\left\lVert #1 \right\rVert}
\newcommand*{\Poissonbracket}[2]{\{ #1, #2 \}}
\newcommand*{\kronecker}[2]{\delta_{#1, #2}}

\DeclareMathOperator{\card}{card}

\newcommand*{\textd}{\mathrm{d}}
\newcommand*{\textnd}{\mathrm{nd}}
\newcommand*{\strat}{\mathrm{strat}}
\newcommand*{\polspace}[1]{\tilde{\mathcal{M}}_{#1}}
\newcommand*{\degpolspace}[1]{\polspace{\br}^{\mathrm{d}}}
\newcommand*{\ndegpolspace}[1]{\polspace{\br}^{\mathrm{nd}}}
\newcommand*{\confspace}[1]{\mathcal{M}_{#1}}
\newcommand*{\degconfspace}[1]{\confspace{\br}^{\mathrm{d}}}
\newcommand*{\ndegconfspace}[1]{\confspace{\br}^{\mathrm{nd}}}
\newcommand*{\class}[1]{\lbrack #1 \rbrack}
\newcommand*{\momentummap}{\mu}
\newcommand*{\diagonal}[1]{\mu_{#1}}
\newcommand*{\length}[1]{\ell_{#1}}
\newcommand*{\dihedralangle}[1]{\theta_{#1}}
\newcommand*{\face}[1]{\Delta_{#1}}
\newcommand*{\integralsmap}{F}
\newcommand*{\bendingfield}[1]{\bX_{#1}}
\newcommand*{\normalizedbendingfield}[1]{\bB_{#1}}
\newcommand*{\SO}[1]{SO(#1)}
\newcommand*{\horizontaltangent}[2]{T^{\text{hor}}_{#1}#2}
\newcommand*{\orbit}[1]{\mathcal{O}(#1)}
\newcommand*{\smooth}{\mathcal{C}^{\infty}}
\newcommand*{\stratsmoothsections}{\Gamma_\strat^\infty}

\newcommand*{\C}{\mathbb{C}}
\newcommand*{\twoframes}[1]{V_2(\C^{#1})}
\newcommand*{\twograss}[1]{\mathrm{Gr}(2,#1)}
\newcommand*{\Hh}{\mathbb{H}}
\newcommand*{\proper}{\mathrm{proper}}
\DeclareMathOperator{\trace}{Trace}

\title[Singular fibers of the bending flows of 3D polygons]{Singular fibers of the bending flows on the moduli space of 3D polygons}
\author[Damien Bouloc]{Damien Bouloc}
\address{Institut de Math\'ematiques de Toulouse \\ 118, route de Narbonne \\ F-31062 Toulouse Cedex 9 }
\email{damien.bouloc@math.univ-toulouse.fr}
\date{September 26, 2016}

\begin{document}

\maketitle

\begin{abstract}
In this paper, we prove that in	the system of bending flows on the moduli space of polygons with fixed side lengths introduced by Kapovich and Millson, the singular fibers are isotropic homogeneous submanifolds. The proof covers the case where the system is defined by any maximal family of disjoint diagonals. We also take in account the case where the fixed side lengths are not generic. In this case, the phase space is an orbispace, and our result holds in the sense that singular fibers are isotropic orbispaces. In a last part we provide leads in favor of a similar study of the integrable systems defined by Nohara and Ueda on the Grassmannian of 2-planes in $\C^n$.
\end{abstract}

\setcounter{tocdepth}{1}
\tableofcontents

\section{Introduction}

In the theory of integrable Hamiltonian systems, singular fibers of the associated Lagrangian fibrations play a very important role. Indeed, according to the classical Liouville--Mineur theorem, each connected component of a compact regular level set of the momentum map is an invariant Lagrangian torus, called a Liouville torus, on which the system is quasi-periodic.  Moreover, near each Liouville torus there exists a system of action-angle variables in which the foliation by Liouville tori is trivial. But the geometry near singular fibers is not so simple in general, and yet it has to be studied in order to understand the local and global geometrical structure of the system. Of particular importance are the \emph{nondegenerate} singular fibers (those which satisfy some natural nondegeneracy conditions), because most singularities of well-known integrable Hamiltonian systems are of this kind. According to a result of Zung \cite{zung1996symplectic}, there is a topological description of nondegenerate singularities in terms of almost direct products of simplest (corank 1 elliptic, corank 1 hyperbolic and corank 2 focus--focus) singularities. Those singularities have been extensively studied, see e.g. \cite{babelon2015higher,bolsinov2004integrable,bolsinov2014singularities,pelayo2014semiclassical}.

On the other hand, degenerate singularities of integrable Hamiltonian systems can be much more complicated. In particular, degenerate singular fibers are not immersed submanifolds in general. However, there is a particular class of integrable Hamiltonian systems whose singular fibers, even the degenerate ones, still look very nice: they are all isotropic homogeneous submanifolds (or more generally isotropic orbispaces when the phase space itself is a symplectic orbispace). This class of singularities, that might be called \emph{spherical singularities}, is closely related to the so called toric degenerations in algebraic geometry (see e.g. \cite{harada2006symplectic,harada2012integrable}). The classical Gel'fand--Cetlin system introduced by Guillemin and Sternberg \cite{guillemin1983gelfand} is an example of integrable systems in this class. The proof that its singularities are spherical has been made by Alamidinne \cite{alamiddine2009GelfandCeitlin} for the Gel'fand--Cetlin system on $\mathfrak{su}(3)$, and then by Miranda and Zung \cite{mirandazung} for the case of $\mathfrak{su}(n)$.

In this paper, we study another family of integrable Hamiltonian systems with spherical singularities: the so called \emph{bending flows} introduced by Kapovich and Millson \cite{KapMil96} on the moduli space $\confspace{\br}$ of 3D polygons with fixed side lengths $\br=(r_1, \dots, r_n)$, which happens to be a manifold when $\br$ is generic. These moduli spaces of polygons and their bending systems have been studied from various points of views afterwards \cite{cdt2014probability,charles2010quantization,HausKnut97,hmm2011toric,kamiyama2002symplectic,nohara2014toric}. Our results here concern their singular fibers and state that the systems of bending flows on $\confspace{\br}$ are indeed examples of systems with spherical singularities: 
\begin{introTheorem}
	\label{intro:thmA}
	For $\br$ generic, the singular fibers of any system of bending flows on $\confspace{\br}$ are isotropic homogeneous submanifolds of the moduli space $\confspace{\br}$.
\end{introTheorem}

Also, we do not limit ourselves to the case when the side lengths are generic. When those lengths are chosen in such a way that the configuration space fails to be a manifold, it is still possible to work in the category of orbispaces. Using the concepts of tangent space, vector fields and symplectic structure on an orbispace (see e.g. \cite{FOR09,Pflaum01,Pfl03}), we extend the definition of the considered Hamiltonian systems to the non-generic case.
\begin{introProposition}
	When $\br$ is not generic, the moduli space $\confspace{\br}$ is a symplectic orbispace, and the systems of bending flows still make sense.
\end{introProposition}
The proof of Theorem~\ref{intro:thmA} in this paper actually includes the non-generic case, leading to the following more general result: 
\begin{introTheorem}
	Let $\br=(r_1, \dots, r_n)$ be any $n$-tuple of positive numbers. The singular fibers of systems of bending flows carry the same structure (manifold or orbispace) as the moduli space $\confspace{\br}$. Moreover, the symplectic structure defined on $\confspace{\br}$ vanishes on those singular fibers.
\end{introTheorem}

The organization of this paper is as follows. In \S\ref{s:geometry_of_polygons}, we recall the definition of the Hamiltonian system associated to a maximal family of disjoint diagonals on the configuration space of 3D polygons with fixed side lengths, and we describe its singularities.
In \S\ref{s:extension_non-generic}, we give more details about how these definitions extend to the non-generic case when one uses the notion of symplectic orbispace.
In \S\ref{s:structure_of_singular_fibers}, we show that the lifts of singular fibers in the space of polygons are manifolds. This allows us to prove that, after projection to the moduli space of polygons, a singular fiber belongs to the same category as the moduli space containing it (i.e. manifolds or orbispaces).
After that, we prove in \S\ref{s:isotropicness_of_fibers} that the singular fibers are isotropic.
Finally in \S\ref{s:relation_to_grassmannians}, we describe how the systems of bending flows on $\confspace{\br}$ relate to integrable systems on the Grassmannian $\twograss{n}$ defined by Nohara and Ueda~\cite{nohara2014toric} and to the Gel'fand--Cetlin system on $U(n)$. In particular we provide some arguments suggesting that the techniques employed in this paper would also apply to the integrable systems on $\twograss{n}$.

\subsection*{Acknowledgements}
The author would like to thank his advisor Nguyen Tien Zung for suggesting the study of the singularities of the bending flows, and for his guidance and helpful discussions. The author is also thankful to the anonymous referee for useful comments and valuable questions which improved the quality of the paper.

\section{Geometry of polygons in Euclidean space}
\label{s:geometry_of_polygons}

\subsection{Notations}

In this section, we recall some results on the configuration space of polygons in the Euclidean space $\R^3$ established by Kapovich and Millson in \cite{KapMil96} and by Hausmann and Knutson in \cite{HausKnut97}.

Fix $n \geq 4$ and a $n$-tuple of positive numbers $\br=(r_1, \dots, r_n)$. Denote by $\norm{.}$ the usual Euclidean norm on $\R^3$ and let $\sphere{2}$ be the unit sphere for this norm. A polygon in $\R^3$ with side lengths $\br$ is given by its vertices $(p_1, \dots, p_n)$ in $\R^3$, satisfying the length condition
\[ \forall 1 \leq i \leq n,\ \norm{p_{i+1} - p_i} = r_i \]
(with the convention $p_{n+1} = p_1$). Up to translations in $\R^3$, such a polygon is actually uniquely determined by the directions
\[ u^i = \frac{p_{i+1} - p_i}{\norm{p_{i+1} - p_i}} \in \sphere{2} \]
of its edges. That is why the set of $n$-gons in Euclidean space whose edges have lengths $r_1, \dots, r_n$ will be identified with the manifold
\[ \polspace{\br} = \set{\bu = (u^1, \dots, u^n) \in (\sphere{2})^n \mid r_1 u^1 + \cdots + r_n u^n = 0}. \]
Here we will be interested in those polygons up to isometric transformations. We denote by $\confspace{\br}$ the quotient space of $\polspace{\br}$ by the diagonal action of $\SO{3}$.

Define a symplectic form $\omega$ on the Cartesian product $(\sphere{2})^n$ by
\[ \omega = \sum_{i=1}^n r_i \omega_i, \]
where $\omega_i$ is the pull-back by the $i$-th projection of the canonical $\SO{3}$-invariant area form on the sphere $\sphere{2}$. Then the diagonal action of $\SO{3}$ on $(\sphere{2})^n$ is Hamiltonian with respect to this form $\omega$, and the associated momentum map is 
\[ \momentummap(u^1, \dots, u^n) = r_1 u^1 + \cdots + r_n u^n \]
(here we have implicitly identified $\mathfrak{so}(3)^\ast$ with $\R^3$, via the usual mapping $u \in \R^3 \mapsto \mathrm{ad}_u = u \vectprod \cdot \in \mathfrak{so}(3)$, and the isomorphism $(\R^3)^\ast \isomorphic \R^3$ given by the canonical Euclidean structure).
The set of 3D polygons with lengths $(r_1, \dots, r_n)$ is exactly the zero level-set of this momentum map. 

Suppose $\br = (r_1, \dots, r_n)$ is generic, that is to say there is no $(\varepsilon_1, \dots, \varepsilon_n) \in \set{\pm 1}^n$ such that
\[ \sum_{i=1}^n \varepsilon_i r_i = 0. \]
Then the action of $\SO{3}$ on $\momentummap^{-1}(0) = \polspace{\br}$ is free, hence the quotient space $\confspace{\br}$ has a natural manifold structure. Denote by $T_{\bu} \polspace{\br}$ the tangent space at $\bu \in \polspace{\br}$ to the space of polygons in Euclidean space. It is the set of $n$-tuples $\btildeX = (\tilde{X}^1, \dots, \tilde{X}^n) \in (\R^3)^n$ satisfying $\scalarproduct{u^i}{\tilde{X}^i} = 0$ for all $1 \leq i \leq n$, and the infinitesimal closing condition
\[ \sum_{i=1}^n r^i \tilde{X}^i = 0. \]
Because the group $\SO{3}$ is compact, the orbit $\orbit{\bu}$ of the $\SO{3}$-action passing through an element $\bu \in \polspace{\br}$ is a closed submanifold of $\polspace{\br}$. Its tangent space $T_{\bu} \orbit{\bu}$ is the set of all $n$-tuples
\[ (x \vectprod u^1, \dots, x \vectprod u^n) \]
with $x \in \R^3$, where $\vectprod$ stands for the vector cross product.
The pairing $\scalarproduct{\btildeX}{\btildeY} = \sum r_i \scalarproduct{\tilde{X}^i}{\tilde{Y}^i}$ defines a Riemannian metric
on $\polspace{\br}$, and then induces a canonical splitting
\[ T_{\bu} \polspace{\br} = T_{\bu} \orbit{\bu} \oplus \horizontaltangent{\bu}{\polspace{\br}}. \]
We then have a natural identification $T_{\class{\bu}} \confspace{\br} \isomorphic \horizontaltangent{\bu}{\polspace{\br}}$ between the tangent space to the configuration space $\confspace{\br}$ at $\class{\bu}$ and the horizontal component of this splitting.

For $1 \leq i, j \leq n$ such that $i \neq j$, denote by 
\[ \diagonal{i,j}(\bu) = 
\begin{cases}
	r_i u^i + r_{i+1} u^{i+1} + \cdots + r_{j-1} u^{j-1} & \text{if } i < j, \\
	- \diagonal{j, i}(\bu) & \text{if } i > j,
\end{cases}
 \]
the vector going from the $i$-th vertex to the $j$-th vertex of the polygon $\bu \in \polspace{\br}$. If $\abs{i - j} = 1$, then $\diagonal{i,j}(\bu)$ is a side of the polygon $\bu$, else it is a diagonal of $\bu$. Its length depends only on the configuration $\class{\bu}$ of the polygon, so the differentiable map $\tilde{f}_{i,j} : \polspace{\br} \rightarrow \R$ given by
\[ \tilde{f}_{i,j}(\bu) = \frac{1}{2} \norm{\diagonal{i,j}(\bu)}^2 \]
induces a well-defined map $f_{i,j} : \confspace{\br} \rightarrow \R$.

Most definitions and results in this section are adapted without new ideas from \cite{KapMil96}, where the authors mainly work with the \emph{caterpillar configuration} where all the diagonals emanate from the first vertex of the polygon ($i=1$ in our definition).

\begin{proposition}[{Kapovich, Millson \cite[Lemma~3.5]{KapMil96}}]
	\label{p:form_of_bending_vector_fields}
	For all $1 \leq i < j \leq n$, the vector field 
	\[ \btildeX_{i,j}(\bu) = (0, \dots, 0, \diagonal{i,j}(\bu) \vectprod u^i, \dots, \diagonal{i,j}(\bu) \vectprod u^{j-1}, 0, \dots, 0) \]
	satisfies $\diff{\tilde{f}_{i,j}} = \omega(\btildeX_{i,j}, \cdot)$. In particular, its image $\bX_{i,j}$ in $(\confspace{\br}, \omega)$ is the Hamiltonian vector field associated to $f_{i,j}$.
\end{proposition}

\begin{proof}
	Let $\btildeY_\bu = (\tilde{Y}^1, \dots, \tilde{Y}^n) \in T_{\bu} \polspace{\br}$. We have
	\[ \omega_{\bu}(\btildeX_{i,j}(\bu), \btildeY_{\bu}) = \sum_{k=i}^{j-1} r_k \det(u^k, \diagonal{i,j}(\bu) \vectprod u^k, \tilde{Y}^k). \]
	It suffices to apply the vector calculus identities
	\[ \det(a, b, c) = \scalarproduct{a}{b \vectprod c} \quad \text{ and } \quad (a \vectprod b) \vectprod c = \scalarproduct{a}{c}b - \scalarproduct{b}{c}a \]
	and use the fact that $\scalarproduct{u_k}{\tilde{Y}^k} = 0$ to obtain
	\[
		\omega_{\bu}(\btildeX_{i,j}(\bu), \btildeY_\bu) = \sum_{k=i}^{j-1} r_k \scalarproduct{\diagonal{i,j}(\bu)}{\tilde{Y}^k} = \diff{\tilde{f}_{i,j}}(\bu)\btildeY \]
\end{proof}

Geometrically, when $\diagonal{i,j}(\bu)$ is a non-vanishing diagonal of $\bu$, this vector field corresponds via its flow to the bending of the polygon $\bu$ along this diagonal with angular speed $\norm{\diagonal{i,j}(\bu)}$. From now on, we will refer to $\btildeX_{i,j}$ as the \emph{bending vector field} associated to the diagonal $\diagonal{i,j}$. On the subset of $\polspace{\br}$ consisting of polygons $\bu$ such that $\diagonal{i,j}(\bu) \neq 0$, one can divide $\btildeX_{i,j}(\bu)$ by $\norm{\diagonal{i,j}(\bu)}$ and obtain a vector field $\btildeB_{i,j}$, which corresponds to the same bending with unit angular speed. Note that those flows are well defined on the quotient space $\confspace{\br}$, and we denote by $\bendingfield{i,j}$ and $\normalizedbendingfield{i,j}$ the images of $\btildeX_{i,j}$ and $\btildeB_{i,j}$ respectively. For later use, we also introduce the \emph{inverse bending vector field} associated to $d = \diagonal{i,j}$
\[ \btildeX_{i,j}^\text{inv}(\bu) = - (d(\bu) \vectprod u^1, \dots, d(\bu) \vectprod u^{i-1}, 0, \dots, 0, d(\bu) \vectprod u^j, \dots, d(\bu) \vectprod u^n) \]
which corresponds geometrically to the bending which rotates (with inverse orientation) the half of the polygon that $\btildeX_{i,j}$ fixes, and vice versa. Of course, $\btildeX_{i,j}$ and $\btildeX_{i,j}^\text{inv}$ have same image in the moduli space $\confspace{\br}$. Indeed,
\[ \btildeX_{i,j}(\bu) -\btildeX_{i,j}^\text{inv}(\bu) = (\diagonal{i,j}(\bu) \vectprod u^1, \dots, \diagonal{i,j}(\bu) \vectprod u^n) \in T_{\bu} \orbit{\bu}.  \]

Following the definitions in~\cite{KapMil96}, we will say that two diagonal maps $\diagonal{i,j}$ and $\diagonal{p,q}$ are \emph{disjoint} if the corresponding diagonals $\diagonal{i,j}(\buzero)$ and $\diagonal{p,q}(\buzero)$ in a convex planar $n$-gon $\buzero$ do not intersect in the interior of $\buzero$. This condition is necessary to obtain the Poisson-commutativity of the associated maps $(f_{i,j})$, that we will use to define a integrable Hamiltonian system on $(\confspace{\br}, \omega)$.

\begin{proposition}[{Kapovich, Millson \cite[Proposition~3.6]{KapMil96}}]
\label{p:diagonal_maps_Poisson_commute}
	If $\diagonal{i,j}$ and $\diagonal{p,q}$ are two disjoint diagonal maps, then the associated vector fields $\btildeX_{i,j}$ and $\btildeX_{p,q}$ satisfy
	\[ \omega(\btildeX_{i,j}, \btildeX_{p,q}) = 0. \]
	In particular the maps $f_{i,j}$ and $f_{p,q}$ Poisson-commute in $(\confspace{\br}, \omega)$.
\end{proposition}

\begin{proof}
	Without loss of generality, we can assume $i < j$, $p < q$ and $i < p$. For any $\bu \in \polspace{\br}$,
	\[ \omega_{\bu}(\btildeX_{i,j}(\bu), \btildeX_{p,q}(\bu)) = \sum_{k \in I} r_k \det(u^k, \diagonal{i,j}(\bu) \vectprod u^k, \diagonal{p,q}(\bu) \vectprod u^k) \]
	where $I$ is the set of integers $k$ such that $i \leq k \leq j-1$ and $p \leq k \leq q-1$. Using vector calculus identities, we obtain
	\[ \omega_{\bu}(\btildeX_{i,j}(\bu), \btildeX_{p,q}(\bu)) = \sum_{k \in I} r_k \det(\diagonal{i,j}(\bu), \diagonal{p,q}(\bu), u^k). \]
	Now it suffices to remark that if $\diagonal{i,j}$ and $\diagonal{p,q}$ are disjoint, then $I$ is either $\integerinterval{p}{q-1}$ or the empty set. In the second case the right-hand side of the equation is zero, in the first case it can be written as \[ \det(\diagonal{i,j}(\bu), \diagonal{p,q}(\bu), \diagonal{p,q}(\bu)) \] and then it vanishes too.
\end{proof}

Given a family of $n-3$ disjoint diagonal maps $d_1, \dots, d_{n-3}$, define a map $\tilde{\integralsmap} = (\tilde{\integralsmap}_1, \dots, \tilde{\integralsmap}_{n-3}) : \polspace{\br} \rightarrow \R^{n-3}$ by
\[ \tilde{\integralsmap}_k(\bu) = \frac{1}{2} \norm{d_k(\bu)}^2 = \tilde{f}_{i,j}(\bu), \]
where $1 \leq k \leq n-3$ and $d_k = \diagonal{i, j}$. This map induces a well-defined map $\integralsmap : \confspace{\br} \rightarrow \R^{n-3}$, which is the integrable Hamiltonian system we are interested in. Now we will recall some results established by Kapovich and Millson~\cite{KapMil96}. They prove most of these results in the case where $d_k = \diagonal{1, k}$ for all $1 \leq k \leq n-3$, but they obviously hold for any choice of disjoint diagonals.

\begin{remark}
\label{r:orientation_of_diagonals}
The two diagonals $\diagonal{i,j}$ and $\diagonal{j,i}$ provide the same map $\tilde{f}_{i,j} = \tilde{f}_{j,i}$, so when fixing a family of disjoint diagonals $(d_1, \dots, d_{n-3})$, we can always assume that each $d_k = \diagonal{i_k, j_k}$ satisfies $i_k < j_k$. In other words, the system $F$ does not depend on the orientation of the diagonals $d_1, \dots, d_{n-3}$. That is why a diagonal $d_k$ will be often considered up to orientation with no further precision.
\end{remark}

\subsection{Singular points of the system}

Suppose fixed a family of disjoint diagonals $(d_1, \dots, d_{n-3})$, and let $\integralsmap : \confspace{\br} \rightarrow \R^{n-3}$ be the corresponding integrable Hamiltonian system on $(\confspace{\br}, \omega)$. For $1 \leq i < j < k \leq n$, the \emph{face} of the polygon $\bu \in \polspace{\br}$ between the vertices $i$, $j$ and $k$ is the triple
\[ \face{i,j,k}(\bu) = (\diagonal{i,j}(\bu), \diagonal{j,k}(\bu), \diagonal{k,i}(\bu)). \]
Such a face will be said \emph{adapted} to the system $\integralsmap$ if each component of the triple $\face{i,j,k}(\bu)$ is either a side $\diagonal{i,i}(\bu) = r_i u^i$ of $\bu$, or one of the fixed diagonal $d_1(\bu), \dots, d_{n-3}(\bu)$ (up to orientation, that is $\diagonal{i,j} = \pm d_p$, see Remark~\ref{r:orientation_of_diagonals} above). This obviously depends only on the integers $i,j,k$: the family of adapted faces $(\face{i,j,k})$ is uniquely determined by the choice of disjoint diagonals. Those faces are exactly the ones with constant edge lengths along the fibers $\integralsmap^{-1}(c_1, \dots, c_{n-3})$ of the system.
Adapted faces provide the following characterization for singular points:

\begin{proposition}[{Charles \cite[Theorem~4.1]{charles2010quantization}}]
	\label{p:description_of_singular_points}
	The configuration $\class{\bu} \in \confspace{\br}$ is a singular point of the system $F$ if and only if there exists a face $\face{i,j,k}$ adapted to $F$ such that $\face{i,j,k}(\bu)$ is degenerate (in the sense that its components are linearly dependent).
\end{proposition}

\begin{proof}
	We will see later that when no adapted face $\face{i,j,k}(\bu)$ is degenerate, then the fiber $N = \integralsmap^{-1}(c_1, \dots, c_{n-3})$ containing $\class{\bu}$ is diffeomorphic to $\torus{n-3}$. This implies that $\class{\bu}$ is a regular value of $\integralsmap$. Hence it suffices to prove now that when some adapted face $\face{i,j,k}(\bu)$ is degenerate, $\class{\bu}$ is a singular value of the system.
	
	Suppose first that a component of $\face{i,j,k}(\bu)$ vanishes, say $\diagonal{i,j}(\bu)$. Then necessarily $j > i + 1$, in other words $\diagonal{i,j}$ is a diagonal and not a side. It follows that the Hamiltonian vector field $\btildeX_{i,j}(\bu)$ vanishes, and then by nondegeneracy of $\omega$, so does the differential of $\tilde{f}_{i,j}$ at $\bu$. By definition of being an adapted face, $\tilde{f}_{i,j}$ is precisely a component of the map $\tilde{\integralsmap}$, hence $\bu$ is a singular value of $\tilde{\integralsmap}$. It follows that $\class{\bu}$ is a singular value of $\integralsmap$.
	
	Suppose now that none of the components of $\face{i,j,k}(\bu)$ vanishes. Recall that $n \geq 4$, so at least one of the components of $\face{i,j,k}(\bu)$ is a diagonal of $\bu$. We will distinguish the cases when exactly one, two or three components are diagonals while the other are sides of the polygon.
	
	\begin{enumerate}
		\item Only one component of $\face{i,j,k}(\bu)$ is a diagonal $d_q(\bu)$. Then the sides of $\face{i,j,k}(\bu)$ are either 
		\begin{itemize}
			\item $d_q(\bu) = \diagonal{p,p+2}(\bu)$, $a = r_p u^p$ and $b = r_{p+1} u^{p+1}$, with $1 \leq p \leq n - 2$,
			\item $d_q(\bu) = \diagonal{n-1, 1}(\bu)$, $a = r_{n-1} u^{n-1}$ and $b = r_n u^n$,
			\item $d_q(\bu) = \diagonal{n, 2}(\bu)$, $a = r_n u^n$ and $b = r_1 u^1$.
		\end{itemize}
		The degeneracy of $\face{i,j,k}(\bu)$ implies $d_q(\bu) \vectprod a = d_q(\bu) \vectprod b = 0$. In the first case, this gives
		\[ \btildeX_{p,p+2}(\bu) = (0, \dots, 0), \]
		while in the two other cases the Hamiltonian vector field $\btildeX$ associated to $F_\ell$ can be written
		\[ \btildeX(\bu) = (d_q(\bu) \vectprod u^1, \dots, d_q(\bu) \vectprod u^n) \in T_{\bu} \orbit{\bu} \]
		and then its image $X(\bu)$ vanishes in $T_{\class{\bu}} \confspace{\br}$.
		Geometrically, that corresponds to the fact that the bending flow associated to $d_q(\bu)$ either has no effect on $\bu$, or rotates the whole polygon $\bu$ (and then has no effect on $\class{\bu}$).
		
		\item Two components of $\face{i,j,k}(\bu)$ are diagonals $d_p(\bu)$ and $d_q(\bu)$.
		
		If those two diagonals are $\diagonal{1,\ell}(\bu)$ and $\diagonal{\ell,n}(\bu)$ (with, necessarily, $3 \leq \ell \leq n-2$), then the third side of $\face{i,j,k}$ is $r_n u^n$. The condition of degeneracy implies the existence of $\alpha, \beta \neq 0$ such that $\alpha \diagonal{1,\ell}(\bu) = u^n = \beta \diagonal{\ell, n}(\bu)$. Then we have
		\[ \alpha \btildeX_{1, \ell}(\bu) + \beta \btildeX_{\ell, n} = (u^n \vectprod u^1, \dots, u^n \vectprod u^n) \in T_{\bu}\orbit{\bu}, \]
		therefore $\alpha \bendingfield{1,\ell}(\bu) + \beta \bendingfield{\ell, n}(\bu) = 0$ in $T_{\class{\bu}}\confspace{\br}$. Geometrically, this illustrate the fact that the flows associated to $d_p(\bu)$ and $d_q(\bu)$ are ``almost'' collinear, except they do not bend the same half of the polygon. They become rigorously collinear once we consider the configuration space $\confspace{\br}$.
		
		Now if those two diagonals are $\diagonal{a,b}(\bu)$ and $\diagonal{a,b+1}(\bu)$, then it suffices to remark that the degeneracy condition $\diagonal{a,b+1}(\bu) = \alpha \diagonal{a,b}(\bu) = \beta u^b$ leads to
		\[ \btildeX_{a,b+1} = \alpha \btildeX_{a,b}. \]
		An analogous equality is obtained when $d_p(\bu) = \diagonal{a,b}(\bu)$ and $d_q(\bu) = \diagonal{a+1, b}(\bu)$.
		
		\item The three sides $\diagonal{i,j}(\bu)$, $\diagonal{j,k}(\bu)$, $\diagonal{i,k}(\bu)$ of the face $\face{i,j,k}$ are diagonals of $\bu$. Then $\diagonal{i,k}(\bu) = \alpha \diagonal{i,j}(\bu) = \beta \diagonal{j,k}(\bu)$ implies
		\[ \alpha \btildeX_{i,j}(\bu) + \beta \btildeX_{j,k}(\bu) = \btildeX_{i,k}(\bu). \]
	\end{enumerate}
	In the three cases, we obtain that the Hamiltonian vector fields associated to the maps $F_1, \dots, F_{n-3}$ are linearly dependent at $\class{\bu}$. Equivalently, the differential maps $\diff{F_1}(\class{\bu}), \dots, \diff{F_{n-3}}(\class{\bu})$ are linearly dependent, therefore $\class{\bu}$ is a singular point of $F$.
\end{proof}

\subsection{Global action--angle coordinates on the regular configurations} Denote by $\confspace{\br}^0$ the regular part of $\confspace{\br}$, that is the set of configurations $\class{\bu}$ such that no adapted face $\face{i,j,k}$ is degenerate at $\bu$. This subset of $\confspace{\br}$ is equipped with global action--angle coordinates.

For $1 \leq k \leq n-3$, define $\length{k} : \confspace{\br}^0 \rightarrow \R$ by
\[ \length{k}(\class{\bu}) = \norm{d_k(\bu)} = 2\sqrt{F_k(\bu)}. \]
The diagonals $d_1, \dots, d_{n-3}$ do not vanish on $\confspace{\br}^0$, so $\length{1}, \dots, \length{n-3}$ are smooth functions on $\confspace{\br}^0$. If $d_k = \diagonal{i,j}$, the Hamiltonian vector field associated to $\length{k}$ is the normalized bending vector field $\normalizedbendingfield{i,j}$ defined previously. Its flow is defined by
\[ \psi_{k}^t(\class{\bu}) = \class{ u^1, \dots, u^{i-1}, R_{d_k(\bu)}^t u^i, \dots, R_{d_k(\bu)}^t u^{j-1}, u^j, \dots, u^n } \]
where $R_{d_k(\bu)}^t$ is the rotation of angle $t$ around the axis $d_k(\bu)$. Note that
\[ 4\Poissonbracket{f_p}{f_q} = \Poissonbracket{\length{p}^2}{\length{q}^2} = 4\length{p}\length{q}\Poissonbracket{\length{p}}{\length{q}} \]
so Proposition~\ref{p:diagonal_maps_Poisson_commute} implies the Poisson-commutativity:
\[ \Poissonbracket{\length{p}}{\length{q}} = 0. \]

For $1 \leq k \leq n-3$, the diagonal $d_k$ belongs to the boundaries of exactly two adapted faces $\face{1}$ and $\face{2}$. For $\bu \in \confspace{\br}^0$, denote by $\hat{\theta}_k(\bu)$ the dihedral angle between $\face{1}$ and $\face{2}$, oriented in such a way that $\hat{\theta}_k$ decreases when applying the flow $\psi_k^t$ with positive values of $t$. Then define a map $\dihedralangle{k} : \confspace{\br}^0 \rightarrow \torus{1}$ by
\[ \dihedralangle{k}(\class{\bu}) = \pi - \hat{\theta}_k(\bu). \]
It is defined this way so that the condition $\theta_k(\class{\bu}) = 0$ for all $1 \leq k \leq n - 3$ corresponds to a planar polygon. Lemma~4.5 of~\cite{KapMil96} states that
\[ \Poissonbracket{\dihedralangle{p}}{\dihedralangle{q}} = 0. \]

By the definitions above, we have
\[ \dihedralangle{p}(\psi_q^t(\class{\bu})) = \dihedralangle{p}(\class{\bu}) + t \kronecker{p}{q}, \]
which after differentiation gives the relation
\[ \Poissonbracket{\dihedralangle{p}}{\length{q}} = \kronecker{p}{q}. \]
Therefore $\length{1}, \dots, \length{n-3}, \dihedralangle{1}, \dots, \dihedralangle{n-3}$ are global action--angle coordinates on $\confspace{\br}^0$.

If $N = \integralsmap^{-1}(c_1, \dots, c_{n-3})$ is a fiber of the system where no adapted face vanishes, these coordinates provide a diffeomorphism
\[ N \isomorphic \torus{1} \times \cdots \times \torus{1} = \torus{n-3}, \]
where each $\torus{1}$ component correspond to the bending flow around some diagonal $d_k$.

\section{Extension to the non-generic case}
\label{s:extension_non-generic}

Suppose now that $\br=(r_1, \dots, r_n)$ is not generic, that is
\[ \sum_{i=1}^{n} \varepsilon_i r_i = 0 \]
for some $(\varepsilon_1, \dots, \varepsilon_n) \in \set{\pm 1}^n$. Then there exist polygons $\bu =(u^1, \dots, u^n) \in \polspace{\br}$ such that $u^1, \dots, u^n$ belong to a same line. For example, take $\bu =(\varepsilon_1 u^0, \dots, \varepsilon_n u^0)$ for any $u^0 \in \sphere{2}$. The existence of such degenerate polygons implies that the action of $G = \SO{3}$ on $\polspace{\br}$ is not free anymore. Indeed, if $\bu$ is a degenerate polygon contained in the line $\set{\lambda u^0 \mid \lambda \in \R}$, then its isotropy group $G_{\bu}$ is the set of all rotations of axis $u^0$. The quotient space $\confspace{\br}$ is not a manifold anymore.

However, the configuration space $\confspace{\br}$ still has the structure of a symplectic orbispace in the sense of~\cite{Pfl03}. Indeed, the orbispace atlas on $\confspace{\br}$ consists of the single chart $(\polspace{\br}, \SO{3}, \pi)$, where $\pi : \polspace{\br} \rightarrow \confspace{\br}$ is the canonical projection on the quotient space. We take as a $\SO{3}$-invariant symplectic form on this chart the form $\omega$ defined above. By definition, the smooth maps $f : U \rightarrow \R$ on an open subset $U \subset \confspace{\br}$ are the maps such that $f \circ \pi : \pi^{-1}(U) \rightarrow \R$ is smooth. This is the case in particular for the maps $f_{i,j}$ defined above.

Recall that a symplectic orbispace has a natural stratification into symplectic manifolds:

\begin{proposition}[Pflaum {\cite[\S 1.3, \S 2.4, and Proposition~3.3]{Pfl03}}]
	Let $G$ be a Lie group acting properly on a smooth manifold $\tilde{M}$. Denote by $M=\tilde{M}/G$ the corresponding quotient space. If $x \in \tilde{M}$, denote by $G_x$ the isotropy group of the action at $x$, by $N(G_x)$ its normalizer in $G$, and by $M_{G_x}$ the submanifold of elements in $M$ with same isotropy group. Then:
	\begin{enumerate}
		\item The manifold $\tilde{M}$ admits a natural stratification \emph{by isotropy type}. The strata are the submanifolds consisting of elements of $\tilde{M}$ whose isotropy groups are conjugate to each other.
		\item This stratification induces a stratification of the quotient space $M$. The stratum $S_x$ containing $\class{x} \in M$ is diffeomorphic to the (smooth) quotient space of $M_x$ by the proper and free action of $\Gamma_x = G_x / N(G_x)$.
	\end{enumerate}
	Let $X$ be an orbispace. Fix $x \in X$ and consider a local orbispace chart $(\tilde{U}, G, \pi)$ around $x$.
	\begin{enumerate}
		\setcounter{enumi}{2}
		\item Let $S_x$ be the stratum containing $x$ in the stratification of $U=\tilde{U}/G$ by isotropy type. Then $S_x$ does not depend on the choice of the local chart $(\tilde{U}, G, \pi)$. It follows that $X$ admits a canonical stratification.
		\item Moreover if $X$ is a symplectic orbispace, then every stratum $S_x$ carries the structure of a Poisson manifold in a canonical way.
	\end{enumerate}
\end{proposition}

In our case, this decomposition coincides with the one between degenerate and nondegenerate polygons.

\begin{proposition}
	\label{p:stratification_confspace}
	Let $\br =(r_1, \dots, r_n)$ be non-generic. Then the configuration space $\confspace{\br}$ is a symplectic orbispace whose corresponding stratification is
	\[ \confspace{\br} = \ndegconfspace{\br} \sqcup \degconfspace{\br}, \]
	where $\ndegconfspace{\br}$ (resp. $\degconfspace{\br}$) is the manifold consisting of $\class{\bu} \in \confspace{\br}$ with $\bu$ nondegenerate polygon (resp. with $\bu$ degenerate polygon).
	
	The $\ndegconfspace{\br}$ component is open and dense in $\confspace{\br}$, while the $\degconfspace{\br}$ component is a finite union of points. Each stratum carries in a natural way the structure of a Poisson manifold.
\end{proposition}

\begin{proof}
	Let $g \in \SO{3}$ be different from the identity. The set of elements in $\sphere{2}$ fixed by $g$ is $\set{v_0, -v_0}$, where $v_0 \in \sphere{2}$ spans the axis of the rotation $g$. Then $g$ belongs to the isotropy group $G_{\bu}$ of a polygon $\bu \in \polspace{\br}$ if and only if $\bu$ is a degenerate polygon contained in the axis of $g$. So, the isotropy group of $\bu \in \polspace{\br}$ is
	\[ G_{\bu} = \begin{cases}
	\set{\text{rotations of axis }u^1} &\text{if } \bu \text{ is degenerate}, \\
	\set{\identity{}} &\text{if } \bu \text{ is nondegenerate}.
	\end{cases} \]
	The subgroups of rotations around a fixed axis are conjugate to each other in $\SO{3}$, so the decomposition of $\polspace{\br}$ with respect to the conjugacy classes of the isotropy groups is the partition
	\[ \polspace{\br} = \ndegpolspace{\br} \sqcup \degpolspace{\br} \]
	between nondegenerate and degenerate polygons, leading to the stratification of $\confspace{\br}$ by the sets $\ndegconfspace{\br} = \pi(\ndegpolspace{\br})$ and $\degconfspace{\br} = \pi(\degpolspace{\br})$.
	
	Now remark the following. The set $\ndegpolspace{\br}$ is exactly the set of polygons $\bu$ with trivial isotropy group, so the action of $\SO{3}$ is free on $\ndegpolspace{\br}$ and $\ndegconfspace{\br}$ is exactly the corresponding quotient manifold. If $\bu \in \degpolspace{\br}$ is a degenerate polygon, then $M_{G_{\bu}}$ is the set of degenerate polygons in $\confspace{\br}$ with same direction as $\bu$. It is a subset of 
	\[ \set{ (\varepsilon_1 u^1, \dots, \varepsilon_n u^1) \mid \varepsilon_1 = \pm 1, \dots, \varepsilon_n = \pm 1} \]
	which is finite, so $\ndegconfspace{\br}$ is also finite.
\end{proof}

Denote by $T\confspace{\br}$ the tangent orbibundle of $\confspace{\br}$. It is the orbispace whose atlas contains the single chart $(T\polspace{\br}, G, p)$, where the action of $G$ on $T\polspace{\br}$ is obtained by differentiating the action of $G$ on $\polspace{\br}$, and $p : T\polspace{\br} \rightarrow G \backslash T\polspace{\br}$ is the canonical projection to the quotient space. To a vector orbibundle is naturally associated a stratified vector bundle:

\begin{proposition}[Pflaum {\cite[\S 2.10]{Pfl03}}]
	Let $E$ be a vector orbibundle. Let $(\tilde{E}, G, p)$ be a local orbibundle chart of $E$ and $(\tilde{U}, G, \pi)$ the associated orbispace chart. For $x \in \tilde{U}$, denote by $G_x$ the isotropy subgroup of the action at $x$, and let $\tilde{E}_x^{G_x}$ be the linear subspace of $G_x$-invariant elements of the fiber $\tilde{E}_x$.
	\begin{enumerate}
		\item If $S$ is a stratum of $\tilde{U}$ (in the stratification by isotropy type), then the space
		\[ \tilde{E}_S = \cup_{x \in \tilde{S}} \, \tilde{E}_x^{G_x} \]
		is a smooth vector bundle over $\tilde{S}$, and $\tilde{E}_S / G$ is a smooth vector bundle over $S = \tilde{S}/G$.
		\item These spaces define a stratification of $E^\strat_U = \cup_{S} \tilde{E}_S / G$.
	\end{enumerate}
	We call \emph{stratified vector bundle} associated to $E$ the stratified space 
	\[ E^\strat = \cup_U \, p(E^\strat_U) \subset E. \]
\end{proposition}

In our case, the stratification takes the following simple form.

\begin{proposition}
	The stratified vector bundle associated to the vector orbibundle $T\confspace{\br}$ is given by the stratification
	\[ T\confspace{\br}^\strat = T\ndegconfspace{\br} \sqcup T\degconfspace{\br}. \]
	Moreover, $T\confspace{\br}^\strat$ is dense in $T\confspace{\br}$.
\end{proposition}

\begin{proof}
	Take $(T\polspace{\br}, G, p)$ the single chart of the tangent orbibundle of $\confspace{\br}$, and let $\bu \in \polspace{\br}$. Recall that $\bX = (X^1, \dots, X^n) \in T_{\bu}\confspace{\br}$ satisfies $\scalarproduct{X^i}{u^i} = 0$ for all $1 \leq i \leq n$, and that the action of $g \in \SO{3}$ on $X$ is defined by
	\[ g \cdot \bX = (gX^1, \dots, gX^n). \]
	If $\bu$ is nondegenerate, then $G_{\bu} = \set{\identity{}}$ and hence $T_{\bu}\polspace{\br}^{G_{\bu}} = T_{\bu} \polspace{\br}$. If $\bu$ is degenerate, then $G_{\bu}$ is the subgroup of $\SO{3}$ consisting of rotations around the axis spanned by any $u^i$ (they are all collinear). Because $X^i$ is orthogonal to $u^i$, $gX^i = X^i$ holds if and only if $X^i = 0$, so we have $T_{\bu} \polspace{\br}^{G_{\bu}} = \set{0}$.
	
	Then define the vector bundles $E^\textnd = \cup_{\bu \in \ndegpolspace{\br}} T_{\bu}\polspace{\br}$ over $\ndegpolspace{\br}$ and $E^\textd = \degpolspace{\br} \times \set{0}$ over $\degpolspace{\br}$. Taking the quotient by $\SO{3}$, one obtains the stratification given in the proposition. The fact that $T\confspace{\br}^\strat$ is dense in $T\confspace{\br}$ comes from the fact that the tangent orbibundle of an orbispace is always a reduced orbibundle.
\end{proof}

A smooth section $\bX : \confspace{\br} \rightarrow T\confspace{\br}$ is a smooth stratified section of the tangent orbibundle $T\confspace{\br}$ if there exists a smooth $\SO{3}$-invariant section $\btildeX : \polspace{\br} \rightarrow T\polspace{\br}$ such that 
\[ p \circ \btildeX = \bX \circ \pi. \]
The space $\stratsmoothsections(T\confspace{\br})$ of smooth stratified sections of the tangent orbibundle is a $\smooth(X)$-module. In particular, the vector fields $\btildeX_{i,j}$ on $T\polspace{\br}$ defined above induce smooth stratified sections of the tangent orbibundle, that we denote by $\bendingfield{i,j}$ as before.

Fix $n-3$ disjoint diagonals $(d_1, \dots, d_{n-3})$ and consider the restrictions to the stratum $\ndegconfspace{\br}$ of the functions $F_1, \dots, F_ {n-3} \in \smooth(\confspace{\br})$ defined above. They define a classical integrable system on $\ndegconfspace{\br}$. Indeed, these maps already Poisson-commute pairwise in $\polspace{\br}$ according to Proposition~\ref{p:diagonal_maps_Poisson_commute}, and the description of singular points given by Proposition~\ref{p:description_of_singular_points} holds on $\ndegconfspace{\br}$ with the same proof. Actually, this description even holds on $\confspace{\br}$ because a degenerate polygon is necessarily a singular point of $F=(F_1, \dots, F_{n-3})$, and in the same time all its faces are degenerate. So in this sense, the integrable system $F=(f_1, \dots, f_{n-3})$ on $\confspace{\br}$ extends to the non-generic case using the notion of symplectic orbispace.

\section{Structure of the singular fibers}
\label{s:structure_of_singular_fibers}

The goal of this section is to prove that a singular fiber $N = \integralsmap^{-1}(c_1, \dots, c_{n-3})$ is generically a submanifold of $\confspace{\br}$. To do so, we will first prove that its lift $\tilde{N} = \tilde{\integralsmap}^{-1}(c_1, \dots, c_n)$ is a submanifold of $\polspace{\br}$ diffeomorphic to a product
\[ \tilde{M} = \SO{3} \times \cdots \times \SO{3} \times \torus{1} \times \cdots \times \torus{1} \times \sphere{2} \times \cdots \times \sphere{2}, \]
and such that the action of $\SO{3}$ on $\tilde{N}$ corresponds to the multiplication on the left on the $\SO{3}$ and $\sphere{2}$ components on $\tilde{M}$. Then it will suffice to prove that the resulting quotient space $\SO{3} \backslash \tilde{M}$ is a manifold.

\medskip

Suppose first that none of the diagonals $d_1, \dots, d_{n-3}$ vanishes on $\tilde{N}$ (following~\cite{HausKnut97}, we say that the polygons in the fiber $\tilde{N}$ ---or the fiber itself--- are \emph{prodigal}). Let $\face{i,j,k}$ be an adapted face which is degenerate on $\tilde{N}$ (recall that $\face{i,j,k}(\bu)$ keeps constant side lengths as $\bu$ varies in $\tilde{N}$, hence the degeneracy of $\face{i,j,k}(\bu)$ is independent of the choice of $\bu \in \tilde{N}$). If we cut the polygon $\bu$ along the line segment containing the degenerate face $\face{i,j,k}(\bu)$, we obtain three polygons 
\[ \begin{aligned}
	\bu_1 &= \left( - \frac{\diagonal{i,j}(\bu)}{\norm{\diagonal{i,j}(\bu)}}, u^i, \dots, u^{j-1} \right) \in \polspace{\br^1}, & \br^1 &= (\rho_1, r_i, \dots, r_{j-1}) \\
	\bu_2 &= \left( - \frac{\diagonal{j,k}(\bu)}{\norm{\diagonal{j,k}(\bu)}}, u^j, \dots, u^{k-1} \right) \in \polspace{\br^2}, & \br^2 &= (\rho_2, r_j, \dots, r_{k-1}) \\
	\bu_3 &= \left( \frac{\diagonal{i,k}(\bu)}{\norm{\diagonal{i,k}(\bu)}}, u^k, \dots, u^n, u^1, \dots, u^{i-1}\right) \in \polspace{\br^3}, & \br^3 &= (\rho_3, r_k, \dots, r_{i-1})
\end{aligned} \]
where $\rho_1, \rho_2, \rho_3 \in \set{r_1, \dots, r_n, c_1, \dots, c_{n-3}}$ do not depend on $\bu \in \tilde{N}$ (see Figure~\ref{f:splitting}). Note that some of these polygons might actually be digons. Because the diagonals $d_1, \dots, d_{n-3}$ are disjoint, for each $1 \leq p \leq 3$ such that $\card (\br^p) \geq 4$, they induce a system $\tilde{\integralsmap}_p$ on $\polspace{\br^p}$ such that $\bu_p$ take values in a fiber $\tilde{N}_p$ of $\tilde{\integralsmap}_p$ as $\bu$ varies in $\tilde{N}$. If $\polspace{\br^p}$ is just a space of digons or triangles, we set $\tilde{N}_p = \polspace{\br^p}$ and we consider that it is a ``regular fiber of the system'' (although there is actually no system defined on $\polspace{\br^p}$) in the sense that it is a manifold diffeomorphic to either $\sphere{2}$ (space of digons or degenerate triangles) or $\SO{3}$ (space of nondegenerate triangles). The map
\[ 	\varphi : \bu \in \tilde{N} \longmapsto (\bu_1, \bu_2, \bu_3) \in \tilde{N}_1 \times \tilde{N}_2 \times \tilde{N}_3 \]
is clearly one-to-one, and its image is the set
\begin{equation}
	\label{eq:image_of_cutting}
	S = \set{ (\bu_1, \bu_2, \bu_3) \in \tilde{N}_1 \times \tilde{N}_2 \times \tilde{N}_3 \mid \alpha_1 u_1^1 + \alpha_2 u_2^1 + \alpha_3 u_3^1 = 0 }
\end{equation}
where the triple $(\alpha_1, \alpha_2, \alpha_3) \neq (0, 0, 0)$ is determined by the relation of linear dependence between the sides of $\face{i,j,k}(\bu)$.

\begin{figure}
	\centering
	\def\svgwidth{0.8\textwidth}
		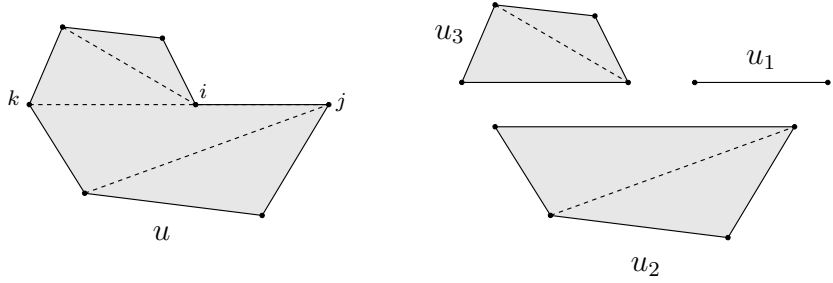
	\caption{Splitting of a singular polygon along a degenerate face}
	\label{f:splitting}
\end{figure}

\begin{proposition}
	\label{p:structure_prodigal_fiber}
	The fiber $\tilde{N}$ is a manifold diffeomorphic to either 
	\begin{itemize}
		\item the sphere $\sphere{2}$,
		\item a product $\SO{3} \times \torus{1} \times \cdots \times \torus{1}$ where each $\torus{1}$ component corresponds to a bending flow on $\tilde{N}$.
	\end{itemize}
\end{proposition}

\begin{proof}
	If $\tilde{N}$ is a fiber consisting in degenerate polygons $\bu$, it is uniquely determined by the first edge $u^1$, and then it is diffeomorphic to the sphere $\sphere{2}$.
	
	Suppose now that $\tilde{N}$ contains nondegenerate polygons. Assume that $\tilde{N}_1$, $\tilde{N}_2$ and $\tilde{N}_3$ are regular fibers. Then $S$ and $\tilde{N}$ are manifolds and $\varphi$ is a diffeomorphism. Define three kinds of actions on $S$:
	\begin{enumerate}
		\item For $g \in \SO{3}$, consider the diagonal action
		\[ g \cdot (\bu_1, \bu_2, \bu_3) = (g \cdot \bu_1, g \cdot \bu_2, g \cdot \bu_3). \]
		\item For every $1 \leq p \leq 3$ such that $\card (\br^p) \geq 4$, define an action of some torus $\torus{q}$ on $S$ using the  bending flows of the system $\tilde{\integralsmap}_p$. Equivalently, it can be defined as the image by $\varphi$ of some bending flow of the system $\tilde{\integralsmap}$. Note that if this flows moves $u_p^1$, we replace it by its inverse flow, which fixes $u_p^1$, so that the action is well-defined on $S$.
		\item For every $\tilde{N}_p$ containing nondegenerate polygons but one, consider the action of $\torus{1}$ on $S$ by rotation of $\bu_p$ around its first edge $u_p^1$ (e.g. if only one $\tilde{N}_p$ contains nondegenerate polygons then there is no action of this kind, if only $\tilde{N}_1$ and $\tilde{N}_2$ contains nondegenerate polygon then consider only the rotation of $\bu_2$, etc.). It is also the image by $\varphi$ of some (possibly inverse) bending flow of $\tilde{\integralsmap}$.
	\end{enumerate}
	All those actions are commuting pairwise, and therefore induce an action of some group
	\[ G = \SO{3} \times \torus{r + \ell} \]
	on $S$, with $r \geq 0$ and $0 \leq \ell \leq 2$, such that each toric component can be interpreted as some bending flow of the system $\tilde{\integralsmap}$. Let us prove that this action is free and transitive.
	
	Let $x = (g, \theta_1, \theta_2, \dots, \theta_{r + \ell}) \in G$ such that $x \cdot (\bu_1, \bu_2, \bu_3) = (\bu_1, \bu_2, \bu_3)$ for some $(\bu_1, \bu_2, \bu_3) \in S$. By definition of our action,
	\[ x \cdot (\bu_1, \bu_2, \bu_3) = (g g_1 \cdot \varphi_1^{t_1}(\bu_1), g g_2 \cdot \varphi_2^{t_2}(\bu_2), g g_3 \cdot \varphi_3^{t_3}(\bu_3) ) \]
	where each $\varphi_p$ is a composition of bending flows, and $g_q$ is either the identity or the rotation with angle $\theta_{r+1}$ or $\theta_{r+2}$ around the axis spanned by $u_{q}^1 \in \sphere{2}$. Hence for each $1 \leq p \leq 3$ we have
	\[ \class{\bu_p} = \class{\varphi_p^{t_p}(\bu_p)} \]
	in the fiber $N_p$ of the moduli space $\confspace{\br^p}$. But on regular fibers, the bending flows act freely so we have $\varphi_p^{t_p} = \identity{}$, or equivalently
	\[ \theta_1 = \cdots = \theta_{r} = 0. \]
	By construction of our action, there is one nondegenerate $\bu_p$ such that $g_p = \identity{}$. Thus we have $g \cdot \bu_p = \bu_p$, which implies $g = \identity{}$. Therefore for all $1 \leq p \leq 3$, either $\bu_p$ is degenerate and then $g_p = \identity{}$ by definition of the action, or $\bu_p$ is nondegenerate and then $g_p \cdot \bu_p = \bu_p$ implies $g_p = \identity{}$. Then $\theta_{r+1}$ and $\theta_{r+2}$, when they exist, vanish. This proves that $x$ is the identity of $G$, so the action is free.
	
	Take now $\bu = (\bu_1, \bu_2, \bu_3)$ and $\bv = (\bv_1, \bv_2, \bv_3)$ in $S$, and let us suitably choose $x \in G$ so that $x \cdot \bu = \bv$. For convenience, let us assume $\ell = 2$ (the proof is similar for $0 \leq \ell \leq 1$). First, using the transitivity of the bending flows on regular fibers, we can fix $\theta_1, \dots, \theta_{r}$ such that
	\[ (\identity{}, \theta_1, \dots, \theta_{r}, 0, 0) \cdot \bu = (\varphi_1^{t_1}(\bu_1), \varphi_2^{t_2}(\bu_2), \varphi_3^{t_3}(\bu_3) ) \]
	satisfies $\class{\varphi_p^{t_p}(\bu_p)} = \class{\bv_p}$ in $N_p$ for all $1 \leq p \leq 3$. In particular, there exists $g \in \SO{3}$ such that $g \cdot \varphi_1^{t_1}(\bu_1) = \bv_1$. Denote by $(\bw_1, \bw_2, \bw_3)$ the triple
	\[ (g, \theta_1, \dots, \theta_{r}, 0, 0) \cdot \bu = (g \cdot \varphi_1^{t_1}(\bu_1), g \cdot \varphi_2^{t_2}(\bu_2), g \cdot \varphi_3^{t_3}(\bu_3) ). \]
	We have $\bw_1 = \bv_1$ and $\class{\bw_q} = \class{g \cdot \varphi_q^{t_q}(\bu_q)} = \class{\varphi_q^{t_q}(\bu_q)} = \class{\bv_q}$ for $2 \leq q \leq 3$. Hence there exists $g_q \in \SO{3}$ such that $g_q \cdot \bw_q = \bv_q$. Recall that elements in $S$ have their components linked by a fixed relation, namely for $q=1,2$, there exists $\varepsilon_q = \pm 1$ such that $u_1^1 = \varepsilon_q u_q^1$ and $v_1^1 = \varepsilon_q v_q^1$. On one hand, the flow $\varphi_q$ preserves the first edge, and $g_q \in \SO{3}$ preserves the orientation, so the first expression implies $w_1^1 = \varepsilon_q w_q^1$. On the other hand, the second expression implies $w_1^1 = \varepsilon_q g_qw_q^1$. Therefore $g_q w_q^1 = w_q^1$, which implies that $g_q$ is a rotation of axis $w_q^1$ with some angle $\alpha_q$ (possibly equal to $0$). Then we have 
	\[ (g, \theta_1, \dots, \theta_{r}, \alpha_2, \alpha_3) \cdot \bu = (\bw_1, g_2 \cdot \bw_2, g_3 \cdot \bw_3) = (\bv_1, \bv_2, \bv_3). \]
	Thus the action of $G$ on $S$ is transitive.
	
	This proves the proposition in the case where $\tilde{N}_1$, $\tilde{N}_2$ and $\tilde{N}_3$ are regular. We then extend it to the general case by induction on the number of degenerate faces of $\tilde{N}$.
\end{proof}

Suppose now that some diagonal $d_k = \diagonal{i,j}$ vanishes on $\tilde{N}$. Then the polygon $\bu$ can be seen as the wedge sum of two polygons with fewer sides
\[ \begin{aligned}
	\bu_1 &= (u^1, \dots, u^{i-1}, u^j, \dots, u^n) \in \polspace{\br^1}, & \br^1 &= (r_1, \dots, r_{i-1}, r_j, \dots, r_n), \\
	\bu_2 &= (u^i, \dots, u^{j-1}) \in \polspace{\br^2}, & \br^2 &= (r_i, \dots, r_{j-1}).
\end{aligned} \]
As before, because the diagonals $d_1, \dots, d_{n-3}$ are disjoint, we have two natural systems on $\polspace{\br^1}$ and $\polspace{\br^2}$ such that $(\bu_1, \bu_2)$ belongs to a product of fibers $\tilde{N}_1 \times \tilde{N}_2$ as $\bu$ varies in $\tilde{N}$. Repeating this process of splitting, we obtain a map
\[ \varphi : \bu \in \tilde{N} \longmapsto (\bu_1, \dots, \bu_q) \in \tilde{N}_1 \times \cdots \times \tilde{N}_q \]
one-to-one and onto where each $\tilde{N}_p$ is a prodigal fiber of some smaller system $\tilde{\integralsmap}_p$. Hence
\[ \tilde{N}_1 \times \cdots \times \tilde{N}_q \]
is a manifold and $\varphi^{-1}$ is an embedding, leading to the following proposition:
\begin{proposition}
	\label{p:structure_nonprodigal_fiber}
	Let $\class{\bu} \in \confspace{\br}$ be a singular point of some system $F : \confspace{\br} \rightarrow \R^{n-3}$ defined by a family of disjoint diagonals. Then the fiber $\tilde{N}$ containing $\bu$ is a manifold, diffeomorphic to a product
	\begin{equation}
		\label{eq:nonprodigal_fiber_model}
		\SO{3} \times \cdots \times \SO{3} \times \torus{1} \times \cdots \times \torus{1} \times \sphere{2} \times \cdots \times \sphere{2}
	\end{equation}
	The action of $\SO{3}$ on $\tilde{N}$ correspond by this diffeomorphism to the multiplication on the left on the $\SO{3}$ and $\sphere{2}$ components in the product.
\end{proposition}

\begin{figure}
	\centering
	\def\svgwidth{0.6\textwidth}
	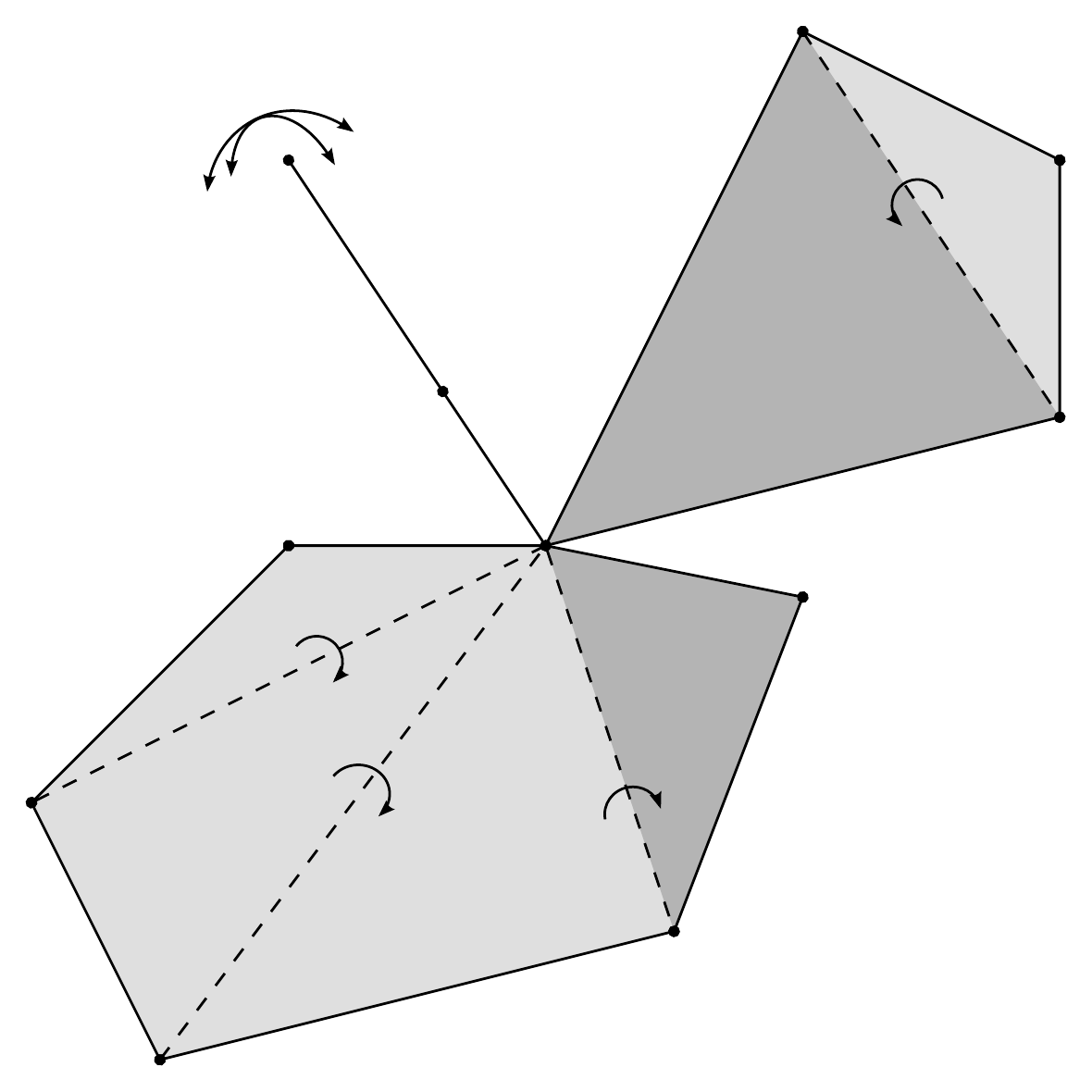
	\caption{Geometrical meaning of $\SO{3}$, $\torus{1}$ and $\sphere{2}$ components}
	\label{f:geometry}
\end{figure}

This decomposition can be explained geometrically as follows (see Figure~\ref{f:geometry}). Consider a polygon $\bu$ in the fiber $\tilde{N}$. If some of its diagonals vanish, it can be seen as a wedge product of prodigal polygons with fewer sides. If one of these smaller polygons is degenerate, one can rotate it around the origin without changing the diagonal lengths of the whole polygon $\bu$: this is the meaning of the corresponding $\sphere{2}$ component. Similarly, a nondegenerate smaller polygon can be rotated around the origin. Once one chooses a face of this polygon as a reference, this rotation is uniquely determined by an element of $\SO{3}$. Applying to $\bu$ those two transformations and the bending flows of each smaller polygon (the $\torus{1}$ components), one can obtain any other polygon in $\tilde{N}$.

\begin{remark}
	It was pointed out to us by the referee that the above decompostion of non-prodigal polygons into wedge sums of prodigal polygons is closely related to the toric manifold constructed by Kamiyama and Yoshida in \cite{kamiyama2002symplectic}. Let us recall briefly their construction. Define an equivalence relation on $\polspace{\br}$ (with the caterpillar configuration) by setting $\bu \sim \bv$ if the following two conditions are satisfied:
	\begin{enumerate}
		\item $\bu$ and $\bv$ are in the same fiber of $\tilde{F}$ (so in particular the same diagonals vanish in $\bu$ and $\bm{v}$),
		\item if $\bu=(\bu_1, \dots, \bu_q)$ and $\bv = (\bv_1, \dots, \bv_q)$ are the decompositions of $\bu$ and $\bv$ into prodigal polygons, then there exists $(g_1, \dots, g_q) \in \SO{3}^q$ such that for any $1 \leq i \leq q$ we have $g_i \cdot \bu_i = \bv_i$.
	\end{enumerate}
	Then the quotient $V = \polspace{\br} / \sim$ is a symplectic toric manifold whose momentum map has same image as $F$ in $\R^{n-3}$ and we have a natural projection $p : \confspace{\br} \rightarrow V$.
	
	Let $\tilde{N}$ be a fiber of $\tilde{F}$. Consider the diffeomorphism \ref{eq:nonprodigal_fiber_model} given by the above proposition that identifies a polygon $\bu$ with $(\bm{g}, \bm{\theta}, \bm{v}) \in (\SO{3})^p \times \torus{q} \times (\sphere{2})^k$. From the diagonal actions of $\SO{3}^p$ on itself and of $\SO{3}^k$ on $(\sphere{2})^k$ we define an action of $\SO{3}^{p+k}$ on $\tilde{N}$ by:
	\[ (\bm{h}, \bm{h'}) \cdot \bu = (\bm{h} \cdot \bm{g}, \bm{\theta}, \bm{h'} \cdot \bm{v}) \]
	for any $(\bm{h}, \bm{h'}) \in \SO{3}^p \times \SO{3}^k = \SO{3}^{p+k}$. Then the orbits for this action are exactly the elements in $p(N) \subset V$.
\end{remark}

Now we would like to determine the structure of the corresponding fiber $N$ in the moduli space $\confspace{\br}$. If $\tilde{N}$ contains at least one $\SO{3}$ component, then considering a polygon $\bu \in \tilde{N}$ up to isometric transformation is equivalent to fixing a given face of a nondegenerate polygon forming $\bu$. So the fiber $N$ should be diffeomorphic to the same product as $\tilde{N}$ but with one $\SO{3}$ component removed. However, if the decomposition of $\tilde{N}$ does not contain any $\SO{3}$ component, or equivalently if $\tilde{N}$ contains only polygons which are wedge sums of degenerate polygons, then the structure of $N$ is much less obvious. We will then distinguish those two cases, saying that: 
\begin{itemize}
	\item $\tilde{N}$ is of type I if there is at least one $\SO{3}$ component after reduction,
	\item $\tilde{N}$ is of type II if there are only $\sphere{2}$ components after reduction.
\end{itemize}
Now we can formulate the following result:

\begin{theorem}
	\label{t:singular_fibers_are_manifolds}
	Let $N$ be a singular fiber of the Hamiltonian integrable system $F = (F_1, \dots, F_{n-3})$ on $\confspace{\br}$. Denote by $\tilde{N}$ the corresponding fiber in $\polspace{\br}$.
	\begin{itemize}
		\item If $\tilde{N}$ is of type I, then $N$ is a manifold diffeomorphic to \[ \SO{3} \times \cdots \times \SO{3} \times \torus{1} \times \cdots \times \torus{1} \times \sphere{2} \times \cdots \times \sphere{2}. \]
		In particular it is an homogeneous manifold.
		\item If $\tilde{N}$ is of type II, then $N$ is an orbispace whose associated stratification is 
		\[ N = N^\textnd \sqcup N^\textd, \]
		where $N^\textnd$ (resp. $N^\textd$) is the manifold $N \cap \ndegconfspace{\br}$ consisting in configurations of nondegenerate polygons (resp. the manifold $N \cap \degconfspace{\br}$ consisting in configurations of degenerate polygons).
	\end{itemize}
\end{theorem}

\begin{proof}
	Suppose $\tilde{N}$ is of type I. It is a homogeneous manifold with at least one $\SO{3}$ component. Recall that the action of $\SO{3}$ on
	\[ \tilde{N} \isomorphic \SO{3}^p \times \torus{q} \times (\sphere{2})^k \]
	is given by
	\[ g \cdot (g_1, \dots, g_p, \theta_1, \dots, \theta_q, v_1, \dots, v_k) = (g g_1, \dots, g g_p, \theta_1, \dots, \theta_q, g v_1, \dots, g v_k), \]
	the corresponding quotient space being $N$ by definition. $\SO{3}$ is compact, and the action is free because there is at least one $\SO{3}$ component, on which $g g_1 \neq g_1$ as long as $g \neq \identity{}$. Hence $N$ is a manifold. Let $M = \SO{3}^{p-1} \times \torus{q} \times (\sphere{2})^k$ be the same product as $\tilde{N}$ with one $\SO{3}$ component removed. The map $\tilde{\varphi} : \tilde{N} \rightarrow M$ defined by
	\[ \tilde{\varphi}(g, \theta, v) = ((g_1^{-1}g_2, \dots, g_1^{-1}g_p), \theta, g_1^{-1}v) \]
	is differentiable and onto. Moreover we have $\tilde{\varphi}(g, \theta, v) = \tilde{\varphi}(g', \theta', v')$ if and only if $\class{(g, \theta, v)} = \class{(g', \theta', v')}$ in $N$, so we obtain a diffeomorphism $\varphi : N \rightarrow M$.
	
	Now if $\tilde{N}$ is of type II, the action of $\SO{3}$ is not free anymore. For example if $g \in \SO{3}$ is a non-trivial rotation around some axis $v_0 \in \sphere{2}$, then one has $g \cdot (v_0, \dots, v_0) = (v_0, \dots, v_0)$ even though $g \neq \identity{}$. However, it is still the quotient space of the smooth action of a compact group on a manifold, so $N$ is an orbispace. The decomposition of $\polspace{\br}$ with respect to the isotropy type restricts to a decomposition
	\[ \tilde{N} = \tilde{N}^\textnd \sqcup \tilde{N}^\textd \]
	with $\tilde{N}^\textnd = \ndegpolspace{\br} \cap N$ and $\tilde{N}^\textd = \degpolspace{\br} \cap N$. The quotient of $\tilde{N}^\textnd$ (resp. of $\tilde{N}^\textd$) by the action of $\SO{3}$ can be naturally identified with $\ndegconfspace{\br} \cap N$ (resp. with $\degconfspace{\br} \cap N$), leading to the stratification stated in the theorem.
\end{proof}

\begin{remark}
	Note that $\polspace{\br}$ admits type II fibers only if $\br$ is not generic. Indeed, suppose the polygon $\bu \in \polspace{\br}$ belongs to some type II fiber
	\[ \tilde{N} \isomorphic \sphere{2} \times \cdots \times \sphere{2}. \]
	Then $\bu$ is a wedge sum of degenerate polygons. Up to rotating each component of this wedge sum, we can construct a polygon $u' \in \confspace{\br}$ which is degenerate, so $\br$ is not generic.
\end{remark}

\begin{remark}
	Let $\tilde{\iota} : \tilde{N} \hookrightarrow \polspace{\br}$ be the inclusion of some fiber $\tilde{N}$ in the space of 3D polygons with lengths $\br$. It is a smooth map compatible with the action of $\SO{3}$, so it induces a morphism $\iota : N \rightarrow \confspace{\br}$ of manifolds or orbispaces (depending on whether $\br$ is generic or not). Theorem~\ref{t:singular_fibers_are_manifolds} states that $N$ is a sub-object of $\confspace{\br}$ carrying the same structure.
\end{remark}

\section{Isotropicness of the fibers}
\label{s:isotropicness_of_fibers}

The goal of this section is to prove that any fiber $N$ of the system $F=(F_1, \dots, F_{n-3})$ defined by disjoint diagonals $d_1, \dots, d_{n-3}$ is isotropic, that is that the symplectic structure $\omega$ on $\confspace{\br}$ vanishes on the vectors tangent to $N$. Recall that we have the stratification
\[ \confspace{\br} = \ndegconfspace{\br} \sqcup \degconfspace{\br} \]
with $\ndegconfspace{\br}$ dense open submanifold of $\confspace{\br}$ and $\degconfspace{\br}$ finite union of points (empty when $\br$ is generic). The tangent space $T\confspace{\br}$ contains the dense stratified space
\[ T\confspace{\br}^\strat = T\ndegconfspace{\br} \sqcup T\degconfspace{\br}. \]
Hence it suffices to prove that
\[ \forall \class{\bu} \in N^\textnd,\ \forall X, Y \in T_{\class{\bu}} N^\textnd,\ \omega_{\class{\bu}}(X, Y) = 0, \]
where $\omega$ is the symplectic form induced on $\ndegconfspace{\br}$.
That is why we will use the following abuse of notation throughout this section: for purpose or clarity, we will write $N$ (respectively $\confspace{\br}$, $\tilde{N}$, $\polspace{\br}$) for $N^\textnd$ (respectively $\ndegconfspace{\br}$, $\tilde{N}^\textnd$, $\ndegpolspace{\br}$), as if $\br$ was generic.

\subsection{Generators of the tangent space} As a first step, it will be useful to exhibit, for any polygon $\bu$ in a singular fiber $\tilde{N}$, a family of vectors that spans the tangent space $T_{u} \tilde{N}$.

For $1 \leq i < j \leq n$ and $v \in \R^3$, set
\[ \btildeY_{i,j}^v(\bu) = (0, \dots, 0, v \vectprod u^i, \dots, v \vectprod u^{j-1}, 0, \dots, 0). \]
Recall that, for $\btildeY_{i,j}^v(\bu)$ to be a well-defined vector in $T_{\bu} \polspace{\br}$, the infinitesimal closing condition has to be verified, namely
\[ \sum_{\ell = i}^{j-1} r_\ell ( v \times u^\ell ) = v \times \diagonal{i,j}(\bu) = 0. \]
Note that this condition is automatically satisfied when $v = \diagonal{i, j}(\bu)$, and the vector obtained is exactly the image at $\bu$ of the bending vector field associated to $\diagonal{i, j}$.

The vector $\btildeY_{i,j}^v(\bu)$ is also well-defined when $\diagonal{i,j}(\bu) = 0$. Therefore if $\bu$ is the wedge sum of proper polygons $\bu_1, \dots, \bu_q$ (as in \S\ref{s:structure_of_singular_fibers}), then in particular we can define vectors $\btildeY_{\bu_1}^v, \dots, \btildeY_{\bu_q}^v$ corresponding to the rotation of each component of the wedge sum around the axis $v \in \R^3$.

\begin{lemma}
	\label{l:generators_of_a_singular_fiber}
	Let $\tilde{N}$ be a singular fiber of the system $\tilde{\integralsmap}$ defined by a family of disjoint diagonals $d_1, \dots, d_{n-3}$. Let $\bu \mapsto (\bu_1, \dots, \bu_q)$ be the decomposition of polygons in $\tilde{N}$ into wedge sums of prodigal polygons.
	
	For every $1 \leq j \leq q$, let $(v_{j,1}, v_{j,2}, v_{j,3})$ be a basis of $\R^3$. Then for every $\bu \in \tilde{N}$, the family
	\[ \set{ \btildeX_i, \btildeY_{\bu_j}^{v_{j,k}}(\bu) \mid 1 \leq i \leq n - 3,\ 1 \leq j \leq q,\ 1 \leq k \leq 3 } \]
	spans the tangent space $T_{\bu} \tilde{N}$.
\end{lemma}

\begin{proof}
	According to \S\ref{s:structure_of_singular_fibers}, $\tilde{N}$ is diffeomorphic to a product
	$\tilde{N}_1 \times \cdots \times \tilde{N}_q$
	where each component of the product satisfies
	\[ \tilde{N}_i \isomorphic
	\begin{cases}
		\sphere{2} &\text{if } \bu_i \text{ is degenerate}, \\
		\SO{3} \times \torus{k_i} &\text{if } \bu_i \text{ is nondegenerate}.
	\end{cases} \]
	Fix $\bu \in \tilde{N}$. For $1 \leq i \leq q$, denote by $\pi_i$ the projection from $\tilde{N}$ onto $\tilde{N}_i$.
	
	If $\tilde{N}_i \isomorphic \sphere{2}$, then the diffeomorphism is provided by a map
	\[ \varphi^i : v \in \sphere{2} \mapsto (\varepsilon_{1} v, \dots, \varepsilon_{n_i} v) \in \tilde{N}_i \]
	with $\varepsilon_j \in \set{\pm 1}$. If $\bu_i = \varphi^i(v)$, the tangent space
	\[ T_v \sphere{2} = \set{ X \in \R^3 \mid \scalarproduct{X}{v} = 0 } \]
	is identified with the set 
	$\set{ X \vectprod v,\  X \in \R^3 }$
	which is the quotient of $\R^3$ by the relation $X \equiv X'$ if $X - X' = \alpha v$ for some $\alpha \in \R$.
	Under this identification, the push-forward $\varphi^i_\ast : \R^3 \rightarrow T_{\bu_i} \tilde{N}_i$ is defined by
	\[ \varphi^i_\ast(X) = X \vectprod \bu_i = (X \vectprod u_i^1, \dots, X \vectprod u_i^{n_i}). \]
	
	If $\tilde{N}_i \isomorphic \SO{3} \times \torus{k_i}$, recall that a diffeomorphism $\varphi^i : \SO{3} \times \torus{k_i} \rightarrow \tilde{N}_i$ is provided by
	$\varphi^i(g, t_1, \dots, t_{k_i}) = (g, t_1, \dots, t_{k_i}) \cdot \bu_i$
	where the action considered in the right-hand term of the expression above is the one defined in \S\ref{s:structure_of_singular_fibers}. In particular,
	$\varphi^i(g, 0, \dots, 0) = g \cdot \bu_i$.
	By the identification $T_{\identity{}}\SO{3} = \R^3$, we have for all $X \in \R^3$,
	\[ \varphi^i_\ast (X, 0, \dots, 0) = X \vectprod \bu_i, \]
	while $\varphi^i_\ast (0, \dots, 0, 1, 0, \dots, 0)$ is some normalized bending flow of the polygon $\bu_i$.
	
	Hence, under the diffeomorphism $\varphi = (\varphi^1 \circ \pi_1, \dots, \varphi^q \circ \pi_q)$ identifying $\tilde{N}$ with a product of $\SO{3}$, $\sphere{2}$ and $\torus{1}$, the image of a vector tangent to a $\torus{1}$ component is collinear to some bending flow $\btildeX_k(\bu)$, and a vector tangent to a $\SO{3}$ or $\sphere{2}$ component is mapped to some vector $\btildeY_{\bu_i}^{X}(\bu)$, $X \in \R^3$. Decomposing $X$ in the basis $(v_{i,1}, v_{i,2}, v_{i,3})$, this vector can be expressed as a linear combination of $\btildeY_{\bu_i}^{v_{i,1}}(\bu)$, $\btildeY_{\bu_i}^{v_{i,2}}(\bu)$ and $\btildeY_{\bu_i}^{v_{i,3}}(\bu)$.
\end{proof}

\subsection{Fibers without vanishing diagonals}

First we suppose that the fixed diagonals $d_1, \dots, d_{n-3}$ do not vanish on $N$, and we prove the isotropicness of $N$ by recursion on the number of degenerate adapted faces on $N$. More precisely, we approximate elements of $\tilde{N}$ by elements in different polygon spaces, belonging to fibers with a lower number of degenerate adapted faces.

\begin{lemma}
	\label{l:approx1}
	Let $N$ be a prodigal fiber of $\integralsmap$, and suppose some adapted face $\face{i,j,k}$ is degenerate on $N$.
	
	Then for any $\buzero \in \tilde{N}$, there exists a neighborhood $I$ of zero in $\R$, a sequence $(\bu_t)_{t \in I}$ of polygons in $\R^3$ and a sequence $(\br^t)_{t \in I}$ of positive side lengths such that:
	\begin{enumerate}[label*=(\thelemma.\arabic*)]
		\item \label{l:approx1:u_t_has_lengths_r} the polygon $\bu_t$ belongs to the space $\polspace{\br^t}$,
		\item \label{l:approx1:r_t_tends_to_r} $\br^t$ tends to $\br$ in $(\R_{>0})^n$ as $t$ tends to zero,
		\item \label{l:approx1:u_t_tends_to_u_0} $\but$ tends to $\buzero$ in $(\sphere{2})^n$ as $t$ tends to zero,
		\item \label{l:approx1:face_no_more_degenerate} for all $t \in I$, $t \neq 0$, the face $\face{i,j,k}(\but)$ is nondegenerate,
		\item \label{l:approx1:no_new_degenerate_face} if some face $\face{a,b,c}(\buzero)$ is nondegenerate, then $\face{a,b,c}(\but)$ is nondegenerate for any $t \in I$,
		\item \label{l:approx1:u_t_is_prodigal} for any $t \in I$, $\bu_t$ is a prodigal polygon.
	\end{enumerate}
	Moreover, if we denote by $\tilde{N}_t$ the fiber containing $\but$ for the function $\tilde{F}_t$ defined on $\polspace{\br^t}$ by the same choice of diagonal $d_1, \dots, d_{n-3}$ as for $\tilde{F}$, then:
	\begin{enumerate}[label*=(\thelemma.\arabic*), resume]
		\item \label{l:approx1:approximation_tangent_vectors} for any $\btildeX \in T_{\buzero} \tilde{N}$, there exists a sequence $(\btildeX_t)_{t \in I}$ that converges to $\btildeX$ in $\R^n$ as $t$ tends to zero and such that for any $t \in I$, $\btildeX_t \in T_{\but} \tilde{N}_t$.
	\end{enumerate}
\end{lemma}

\begin{proof}
	We construct these sequences explicitly. Fix $x \in \sphere{2}$ a vector orthogonal to $\diagonal{i,j}(\buzero)$ and set
	\[ \but = \left( u_0^1, \dots, u_0^{j-2}, \frac{r_{j-1}u_0^{j-1} + t x}{\norm{r_{j-1}u_0^{j-1} + t x}}, \frac{r_{j}u_0^{j} - t x}{\norm{r_{j}u_0^{j} - t x}}, u_0^{j+1}, \dots, u_0^n \right) \]
	and
	\[ \br^t = (r_1, \dots, r_{j-2}, \norm{r_{j-1}u_0^{j-1} + t x}, \norm{r_{j}u_0^{j} - t x}, r_{j+1}, \dots, r_n). \]
	Geometrically, the polygon $u_t$ is obtained by moving the $j$-th vertex of $\buzero$ in the direction $x \in \sphere{2}$ as illustrated in Figure~\ref{f:approx1}. Properties~\ref{l:approx1:u_t_has_lengths_r}, \ref{l:approx1:r_t_tends_to_r} and \ref{l:approx1:u_t_tends_to_u_0} are straightforward.
	
	For Property~\ref{l:approx1:face_no_more_degenerate}, remark that $\diagonal{i,k}(\but) = \diagonal{i,k}(\buzero)$ but 
	\[ \diagonal{i,j}(\but) = \diagonal{i,j}(\buzero) + t x. \]
	As $x \neq 0$ is orthogonal to $\diagonal{i,j}(\buzero) \neq 0$, we obtain that $\diagonal{i,j}(\but)$ is no more collinear to $\diagonal{i,k}(\but)$ when $t \neq 0$. For Property~\ref{l:approx1:no_new_degenerate_face}, we use the fact that the map $t \mapsto \diagonal{a,b}(\but) \vectprod \diagonal{a,c}(\but)$ is continuous, so if $\diagonal{a,b}(\buzero)$ and $\diagonal{a,c}(\buzero)$ are linearly independent, then $\diagonal{a,b}(\but)$ and $\diagonal{a,c}(\but)$ are linearly independent for any $t$ in some neighborhood of zero. The same argument is used for Property~\ref{l:approx1:u_t_is_prodigal}.
	
	Finally, for Property~\ref{l:approx2:approximation_tangent_vectors}, it sufficed to show that any vector in the family of generators given by Lemma~\ref{l:generators_of_a_singular_fiber} can be approximated as claimed. For the bending vector fields, this comes from the fact that the map $t \mapsto \btildeX_{i}(\but) \in \R^n$ is continuous for any $1 \leq i \leq n-3$. The same argument is used for $\btildeY^v_{1,n}(\buzero)$ once one remarks that $\btildeY^v(\but)$ is well-defined for any $t \in I$.
\end{proof}

\begin{figure}
	\centering
	\def\svgwidth{0.9\textwidth}
	{\tiny
		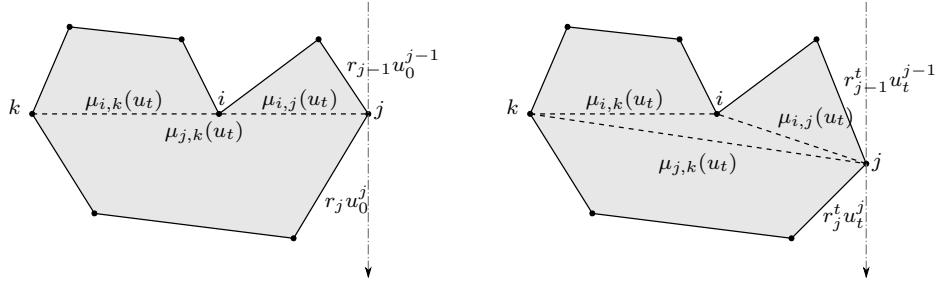 }
	\caption{Approximation of a polygon with a degenerate face}
	\label{f:approx1}
\end{figure}

\begin{proposition}
	\label{p:isotropicness_of_prodigal_fibers}
	For any side lengths $\br \in (\R_{>0})^n$ and any choice of diagonals on $\confspace{\br}$, the prodigal fibers of the associated integrable system are isotropic.
\end{proposition}

\begin{proof}
	We prove it by recursion on the number of degenerate face on the fiber $N$. If there are no degenerate face, then the fiber is regular and hence it is Lagrangian.
	
	Suppose now that $m > 0$ adapted faces are degenerate on $N$. Take $\buzero$ in $\tilde{N}$ and consider the approximation $(\but)_t$ of $\buzero$ provided by Lemma~\ref{l:approx1}. By Properties~\ref{l:approx1:face_no_more_degenerate} and \ref{l:approx1:no_new_degenerate_face}, the polygon $\but$ has at most $m-1$ degenerate faces when $t \neq 0$ and then the fiber $N_t$ containing $\class{\but}$ is isotropic.
	
	Let $\btildeX_1, \btildeX_2 \in T_{\buzero} \tilde{N}$, and $(\btildeX_{1,t})_t, (\btildeX'_{2,t})_t$ their approximations provided by Property~\ref{l:approx1:approximation_tangent_vectors}. Denote by $\omega^t$ the symplectic form on $\confspace{\br_t}$. Recall that it is the restriction of a two-form on $(\sphere{2})^n$ that satisfies
	\[ \omega^t_{\but}(\btildeX_{1,t}(\but), \btildeX_{2,t}(\but)) = \sum_{i \in I(p,q)} r^t_i \det(u_t^i, X_{1,t}^i, X_{2,t}^i) \]
	where $I(p,q)$ is a subset of $\integerinterval{1}{n-3}$ uniquely determined by the choice of diagonals $d_1, \dots, d_k$ (see proof of Proposition~\ref{p:diagonal_maps_Poisson_commute}). It follows that
	\[ \lim\limits_{t \to 0} \omega^t_{\but}(\btildeX_{1,t}(\but), \btildeX_{2,t}(\but)) = \omega_{\buzero}(\btildeX_1, \btildeX_2) \]
	Since $N_t$ is isotropic for $t \neq 0$ we have $\omega_{\class{\buzero}}(\bX_1, \bX_2) = 0$.
\end{proof}

\subsection{Fibers with vanishing diagonals}

We now prove the isotropicness in the general case, assuming that some of the disjoint diagonals $d_1, \dots, d_{n-3}$ vanish on $N$. We will prove the result by recursion on the number of vanishing diagonals.

\begin{lemma}
	\label{l:approx2}
	Let $N \subset \confspace{\br}$ be a singular fiber of $F$, and $\tilde{N}$ its lift in $\confspace{\br}$. If some diagonal $d_k$ vanishes on $\tilde{N}$, then there exists a dense subset $\tilde{S} \subset \tilde{N}$ such that for any $\buzero \in \tilde{S}$, there exists a neighborhood $I$ of zero in $\R$ and a sequence of polygons $(\but)_{t \in I}$ in $\confspace{\br}$ such that
	\begin{enumerate}[label*=(\thelemma.\arabic*)]
		\item \label{l:approx2:u_t_tends_to_u_0} $\but$ tends to $\buzero$ as $t$ tends to zero,
		\item \label{l:approx2:diagonal_no_more_zero} for all $t \in I$, $t \neq 0$, $d_k(\but) \neq 0$,
		\item \label{l:approx2:no_new_zero_diagonal} for all $1 \leq \ell \leq n - 3$, if $d_\ell(\buzero) \neq 0$ then $d_\ell(\but) \neq 0$ for all $t \in I$,
	\end{enumerate}
	Moreover, if we denote by $\tilde{N}_t$ the fiber of $\tilde{F}$ containing $\but$, then:
	\begin{enumerate}[label*=(\thelemma.\arabic*), resume]
		\item \label{l:approx2:approximation_tangent_vectors} for any $\btildeX \in T_{\buzero} \tilde{N}$, there exists a sequence $(\btildeX_t)_{t \in I}$ that converges to $\btildeX$ in $\R^n$ as $t$ tends to zero and such that for any $t \in I$, $\btildeX_t \in T_{\buzero} \tilde{N}_t$.
	\end{enumerate}
\end{lemma}

\begin{proof}
	Let $\buzero = (\bu_{0,1}, \cdots, \bu_{0,q})$ be the decomposition of $\buzero$ into prodigal polygons. Up to a change of indices, we can assume that this decomposition is given by a sequence
	\[ 1 = p_0 < p_1 < \cdots < p_q = n + 1 \]
	such that $\bu_i = (u_0^{p_{i-1}}, \cdots, u_0^{p_i - 1})$ for any $0 \leq i \leq q$ and $d_k = \diagonal{p_0, p_1}$.
	
	The diagonal $d_k$ is the side of exactly two adapted faces $\face{p_0, k_1, p_1}$ and $\face{p_0, p_1, k'_2}$ with $1 < k_1 < p_1 < k'_2$.
	Suppose $k'_2 = p_j$ for some $j \geq 2$. Then $\diagonal{k'_2, k_1} = \diagonal{p_j, p_1}$ is the side of two adapted faces: one is $\face{p_0, p_1, k'_2}$ and the other is $\face{p_1, k''_2 , p_j}$ for some $p_1 < k''_2 < p_j$. Now $k''_2$ may be equal to $p_{j'}$ for some $2 < j' < j$, but iterating the previous construction, we obtain after a finite number of steps a sequence
	\[ 1 < k_1 < p_1 < k_2 \]
	such that $k_2 \neq p_j$ for all $1 \leq j \leq q$ and $\face{p_1, k_2, p_{j_0}}$ is an adapted face for some $j_0$.
	
	Define the subset $\tilde{S} = \{ \bu \in \tilde{N} \mid \diagonal{k_1, p_1}(\bu) \vectprod \diagonal{p_1, k_2}(\bu) \neq 0 \} \subset \tilde{N}$ and for fixed $\buzero \in \tilde{S}$, set
	\[ \but = (u_0^1, \dots, u_0^{p_1-2}, R^t u_0^{p_1-1}, R^t u_0^{p_1}, u_0^{p_1+1}, \dots, u_0^n) \]
	where $R^t$ is the rotation of angle $t$ around the axis
	\[ \diagonal{k_1,k_2}(\buzero) = r_{k_1} u_0^{k_1} + \cdots + r_{k_2 - 1} u_0^{k_2 - 1}. \]
	Remark that the family of polygon $\but$ is geometrically obtained by bending the polygon $\buzero$ along its diagonal $\diagonal{k_1, k_2}(\buzero)$, as illustrated in Figure~\ref{f:approx2}. From this definition Property~\ref{l:approx2:u_t_tends_to_u_0} is immediate and Property~\ref{l:approx2:no_new_zero_diagonal} follows from continuity of the map $t \mapsto d_\ell(\but) \in \R^3$. 
	
	The diagonals of $\but$ satisfy
	\[ \diagonal{p, q}(\but) = \diagonal{p, q}(\buzero) + \sum_{i \in I} r_i(R^t u_0^i - u_0^i) \]
	where $I = \integerinterval{p}{q - 1} \cap \integerinterval{k_1}{k_2 - 1}$.
	In particular, this implies
	\[ \diagonal{p_0, p_1}(\but) = 0 + \sum_{i=k_1}^{p_1 - 1} r_i(R^t u_0^i - u_0^i) = R^t \diagonal{k_1, p_1}(\buzero) - \diagonal{k_1, p_1}(\buzero). \]
	Since $\diagonal{k_1, p_1}(\buzero) \vectprod \diagonal{k_1, k_2}(\buzero) =  \diagonal{k_1, p_1}(\buzero) \vectprod \diagonal{p_1, k_2}(\buzero)$ does not vanish for $\buzero \in \tilde{S}$, the rotation $R^t$ does not act trivially on $\diagonal{k_1, p_1}(\buzero)$, whence Property~\ref{l:approx2:diagonal_no_more_zero}.
	
	Finally for Property~\ref{l:approx2:approximation_tangent_vectors} it suffices to show that we can approximate any vector among the generators given in Lemma~\ref{l:generators_of_a_singular_fiber}. It is clear that for any $1 \leq \ell \leq n-3$,
	\[ \lim\limits_{t \to 0} \btildeX_\ell(\but) = \btildeX_\ell(\buzero). \]
	The decomposition of $\but$, $t \neq 0$, into prodigal polygons is given by the sequence
	\[ 1 = p_0 < p_2 < p_3 < \cdots < p_q = n + 1. \]
	For any $3 \leq i \leq q$ and for any $v \in \R^3$, the vector $\btildeY^v_{p_{j-1},p_j}(\but)$ is tangent to the fiber $\tilde{N}_t$ containing $\but$ and we have
	\[ \lim\limits_{t \to 0} \btildeY^v_{p_{j-1},p_j}(\but) = \btildeY^v_{p_{j-1},p_j}(\buzero). \]
	For $v_1 = \diagonal{k_1, p_1}(\buzero)$ and $v_2 = \diagonal{p_1, k_2}(\buzero)$ we have
	\begin{align*}
		\btildeY^{v_1}_{p_0, p_1}(\buzero) &= \lim\limits_{t \to 0} (\btildeY_{k_1, p_1}^{\diagonal{k_1, p_1}(\but)}(\but) - \btildeY_{p_0, k_1}^{\diagonal{p_0, k_1}(\but)}(\but)), \\
		\btildeY^{v_1}_{p_1, p_2}(\buzero) &= \lim\limits_{t \to 0} (\btildeY_{k_1, p_2}^{\diagonal{k_1, p_2}(\but)}(\but) - \btildeY_{k_1, p_1}^{\diagonal{k_1, p_1}(\but)}(\but)), \\
		\btildeY^{v_2}_{p_0, p_1}(\buzero) &= \lim\limits_{t \to 0} (\btildeY_{p_0, k_2}^{\diagonal{p_0, k_2}(\but)}(\but) - \btildeY_{p_1, k_2}^{\diagonal{p_1, k_2}(\but)}(\but)), \\
		\btildeY^{v_2}_{p_1, p_2}(\buzero) &= \lim\limits_{t \to 0} (\btildeY_{p_1, k_2}^{\diagonal{p_1, k_2}(\but)}(\but) - \btildeY_{k_2, p_2}^{\diagonal{k_2, p_2}(\but)}(\but)),
	\end{align*}
	where in each expression the right-hand side is a limit of well-defined vectors tangent to the fibers $\tilde{N}_t$. Since $\buzero \in \tilde{S}$, the vectors $v_1$ and $v_2$ are linearly independent and together with
	\[ v_3 = \frac{v_1 \vectprod v_2}{\norm{v_1 \vectprod v_2}} \]
	they form a basis of $\R^3$. Note that
	\[ \begin{cases}
		\diagonal{p_0, p_1}(\but) = R^t x_1 - x_1 &\text{where } x_1 = \diagonal{k_1, p_1}(\buzero), \\
		\diagonal{p_1, p_2}(\but) = R^t x_2 - x_2 &\text{where } x_2 = \diagonal{p_1, k_2}(\buzero).
	\end{cases} \]
	Since $x_i \vectprod \diagonal{k_1, k_2}(\buzero) = \pm \diagonal{k_1, p_1}(\buzero) \vectprod \diagonal{p_1, k_2}(\buzero) = \pm v_1 \vectprod v_2$, we have
	\[ \lim\limits_{t \to 0} \frac{R^t x_i - x_i}{\norm{R^t x_i - x_i}} = \pm v_3 \]
	and thus the normalized bending vector fields associated to $\diagonal{p_0, p_1}(\but)$ and $\diagonal{p_1, p_2}(\but)$ converge to $\pm \btildeY^{v_3}_{p_0, p_1}(\buzero)$ and $\pm \btildeY^{v_3}_{p_1, p_2}$ respectively.
\end{proof}

\begin{figure}
	\centering
	\def\svgwidth{0.75\textwidth}
	{\tiny
		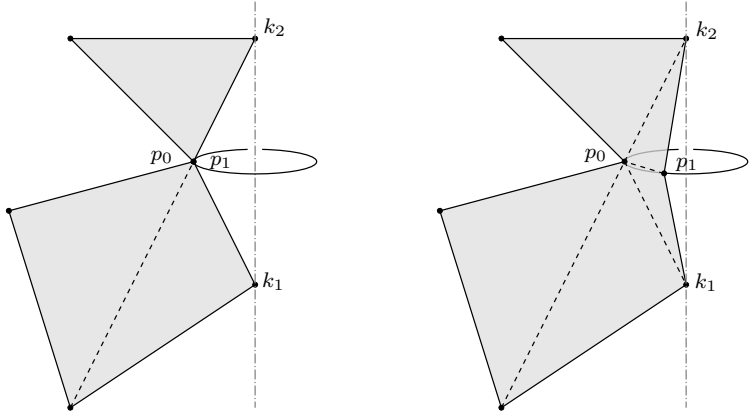 }
	\caption{Approximation of a polygon with a vanishing diagonal}
	\label{f:approx2}
\end{figure}

\begin{theorem}
	\label{t:singular_fibers_are_isotropic}
	Let $F$ be the integrable Hamiltonian system on $(\confspace{\br}, \omega)$ defined by a family of disjoint diagonals $(d_1, \dots, d_{n-3})$. Let $N$ be a singular fiber of $F$. Then $N$ is isotropic.
\end{theorem}

\begin{proof}
	It suffices to prove the result by induction on the number of vanishing diagonals. If no diagonal vanishes, the istropicness follows from Proposition~\ref{p:isotropicness_of_prodigal_fibers}.
	
	If $m > 0$ diagonal vanish, then we approximate any polygon $\buzero$ in $\tilde{S} \subset \tilde{N}$ by a sequence of polygons with at most $m-1$ diagonals using Lemma~\ref{l:approx2} and as in the proof of Proposition~\ref{p:isotropicness_of_prodigal_fibers} we show that the symplectic form vanishes on $T_{\class{\buzero}} N$. Since $\tilde{S}$ is dense in $\tilde{N}$ the result extends to any $\buzero \in \tilde{N}$ and then the fiber $N$ is isotropic.
\end{proof}

Let us conclude this section by characterizing the cases where these isotropic singular fibers have maximal dimension, and therefore are Lagrangian.

\begin{corollary}
	Let $N$ be a singular fiber of $F = (F_1, \dots, F_{n-3})$, and $N^\textnd$ the manifold consisting of the nondegenerate polygons in $N$ (for generic side lengths, $N^\textnd = N$). Consider the decomposition
	\[ \tilde{N} \isomorphic \tilde{N}_1 \times \cdots \times \tilde{N}_q \]
	of $\tilde{N}$ into prodigal fibers.
	
	Then $N^\textnd$ is a Lagrangian manifold if and only if each $\tilde{N}_i$ is either a space of digons, a space of nondegenerate triangles, or a regular fiber of $\tilde{F}_i$.
\end{corollary}

\begin{proof}
	By Theorem~\ref{t:singular_fibers_are_isotropic}, $N^\textnd$ is Lagrangian if and only if it has dimension $n-3$. Therefore we just have to compute the dimension of $N^\textnd$.
	
	Recall that $\tilde{N}_i$ is diffeomorphic to
	\[ 	\begin{cases}
			\sphere{2} &\text{if } \tilde{N}_i \text{ is a space of digons}, \\
			\SO{3} &\text{if } \tilde{N}_i \text{ is a space of nondegenerate triangles}, \\
			\SO{3} \times \torus{n_i - 3} &\text{if } \tilde{N}_i \text{ is a regular fiber of a system on a space of} \\
			& \text{polygons with } n_i \geq 4 \text{ sides}.
		\end{cases} \]
	In each of the above cases, the dimension of $\tilde{N}_i$ is equal to the number of sides $n_i$ of the polygons in $\tilde{N}_i$. Therefore, if each $\tilde{N}_i$ corresponds to one of the above cases, then the product $\tilde{N}$ has dimension $n_1 + \cdots + n_q = n$. It follows that the quotient $N^\textnd$ of the (free) action of $\SO{3}$ on the manifold $\tilde{N}^\textnd$ dense and open in $\tilde{N}$ has dimension $n-3$.
	
	On the other hand, $\tilde{N}_i$ is diffeomorphic to
	\[ 	\begin{cases}
	\sphere{2} &\text{if } \tilde{N}_i \text{ is a space of degenerate } n_i\text{-gons, } n_i \geq 3, \\
	\SO{3} \times \torus{p_i - 3} &\text{with } 0 \leq p_i < n_i \text{ if } \tilde{N}_i \text{ is a singular fiber of a } \\ &\text{system on a space of } n_i\text{-gons}
	\end{cases} \]
	In both cases, we have $\dim \tilde{N_i} < n_i$. Therefore, if such a component appears in the product $\tilde{N}$, we have $\dim N^\textnd < n - 3$.
\end{proof}

\section{Relation to Grassmannians and Gel'fand--Cetlin}
\label{s:relation_to_grassmannians}

\subsection{From Grassmannians to polygon spaces}

In this subsection, we recall the relation described by Hausmann and Knutson~\cite{HausKnut97} between the Grassmannian manifold of $2$-planes in $\C^n$ and the family of polygon spaces with $n$ sides.

Fix $n \geq 3$ and denote by $\twoframes{n}$ the manifold of $2$-frames in $\C^n$, that is the set of pairs $(\bz, \bw)$ of orthogonal unit vectors in $\C^n$, identified with a subspace of $n \times 2$ matrices. The right action of $U(2)$ on $\twoframes{n}$ by matrix multiplication corresponds to the orthogonal transformations of $\C^n$ leaving the plane spanned by $\bz$ and $\bw$ invariant. The quotient manifold
\[ \twograss{n} = \twoframes{2} / U(2) \]
can then be identified as the space of $2$-planes in $\C^n$.

Let $\Hh = \C \oplus j\C$ be the skew-field of quaternions. The Euclidean space $\R^3$ will be identified with the space $I\Hh = i\R \oplus j\R \oplus k\R$ of imaginary quaternions, with inner product induced by the canonical Hermitian structure on $\Hh = \C^2$. A $3$-dimensional polygon (based at the origin) will now be defined as a vector $\bq = (q^1, \dots, q^n) \in (I\Hh)^n$ satisfying the closing condition $q^1 + \cdots + q^n = 0$. Given $\br = (r_1, \dots, r_n) \in (\R_{>0})^n$, the space of 3d polygons with side lengths $\br$ is now defined as the manifold
\[ \polspace{\br} = \lbrace \bq \in (I\Hh)^n \mid q^1 + \cdots + q^n = 0,\ \norm{q^1} = r_1,\ \dots,\ \norm{q^n} = r_n \rbrace. \]
We will also consider the manifold $\polspace{(2)}$ of polygons $\bq$ with perimeter $\abs{\bq} = \norm{q^1} + \cdots + \norm{q^n}$ equal to $2$. Note that, at this point, we haven't excluded \emph{improper} polygons $\bq$, for which some side $q^i$ vanishes. We have
\[ \bigcup_{\br \in (\R_{>0})^n, \abs{\br} = 2} \polspace{\br} = \polspace{(2)}^\proper \subsetneq \polspace{(2)}. \]

Consider the application $\varphi : \Hh \rightarrow I\Hh$ defined by $\varphi(q) = \bar{q} i q$, or equivalently $\varphi(z + jw) = i(\abs{z}^2 - \abs{w}^2 + 2\bar{z}wj)$. It maps the $3$-sphere of radius $\sqrt{r}$ in $\Hh$ onto the $2$-sphere of radius $\br$ in $I\Hh$. Observe that, for any $\bz, \bw \in \C^n$, one has
\[ \sum_{\ell = 1}^{n} \varphi(z^\ell + j w^\ell) = i(\norm{\bz}^2 - \norm{\bw}^2 + 2 \scalarproduct{\bz}{\bw} j). \]
In particular, if $(\bz, \bw) \in \twoframes{n}$, then the $n$-tuple
\[ \tilde{\Phi}(\bz, \bw) = (\varphi(z^1 + jw^1), \dots, \varphi(z^n + jw^n)) \]
defines a polygon in $I\Hh$, with perimeter
\[ \abs{\tilde{\Phi}(\bz, \bw)} = \sum_{\ell = 1}^{n} \norm{\phi(z^\ell + jw^\ell)}_{I\Hh} = \sum_{\ell = 1}^{n} \norm{z^\ell + jw^\ell}^2_{\Hh} = \norm{\bz}^2 + \norm{\bw}^2 = 2. \]
We thus have defined a map $\tilde{\Phi} : \twoframes{n} \rightarrow \polspace{(2)}$ which is onto.

Let $\eta$ be the usual inclusion of $\Hh$ in the space of $2 \times 2$ complex matrices defined by
\[ \eta(z + jw) = \begin{pmatrix}
	z & w \\
	-\bar{w} & \bar{z}
\end{pmatrix}. \]
We define actions of $U(2)$ on $\Hh$ on the left and on the right as the pull-backs by $\eta$ of matrix multiplication (on the left and on the right). For these actions we have the relation: for any $q \in \Hh$ and $P \in U(2)$, 
\[ \varphi(q \cdot P) = P^{-1} \cdot \varphi(q) \cdot P. \]
Note that $\trace(\eta(q)\eta(q')^\ast) = 2\scalarproduct{q}{q'}_{\Hh}$ hence $q \mapsto P^{-1} \cdot q \cdot P$ belongs to the group $SO(I\Hh)$ of orthogonal transformations on $I\Hh$. It follows that $\tilde{\Phi}((\bz, \bw)P)$ lies in the orbit of $\tilde{\Phi}(\bz, \bw)$ for the diagonal action of $\SO{I\Hh}$ on $(I\Hh)^n$, and thus we obtain a well-defined map
\[ \Phi : \twograss{n} \longrightarrow \confspace{(2)} = \polspace{(2)} / \SO{I\Hh}. \]

Denote by $T_{U(n)}$ the maximal torus of diagonal matrices in $U(n)$, acting on $\twoframes{n}$ by multiplication. We have the following: 

\begin{proposition}[Hausmann, Knutson {\cite[Theorem 3.6]{HausKnut97}}]
	The restriction $\tilde{\Phi}^\proper$ of $\tilde{\Phi} : \twoframes{n} \rightarrow \polspace{(2)}$ above the space $\polspace{(2)}^\proper$ of proper polygons is smooth a principal $T_{U(n)}$-bundle.
\end{proposition}

On can check that the action of $T_{U(n)}$ on $\twoframes{n}$ descends to an action on $\twograss{n}$. However this action is no longer effective: its center is the subspace $\Delta \isomorphic \sphere{1}$ of homothetic transformations of $T_{U(n)}$.

\begin{proposition}[Hausmann, Knutson {\cite[Theorem 3.9]{HausKnut97}}]
	\label{p:Phi_is_principal_bundle}
	The restriction $\Phi^\proper$ of $\Phi : \twograss{n} \rightarrow \confspace{(2)}$ above the space $\confspace{(2)}^\proper$ of (classes of) proper polygons is a smooth principal $(T_{U(n)}/\Delta)$-bundle.
\end{proposition}

Actually, the action of $T_{U(n)}$ on $\twograss{n}$ is Hamiltonian, with momentum map $\momentummap_{T_{U(n)}} : \twograss{n} \rightarrow \R^n$ given by:
\begin{align*}
\momentummap_{T_{U(n)}}(\class{\bz, \bw}) &= \left( \frac{\abs{z^1}^2 + \abs{w^1}^2}{2}, \dots, \frac{\abs{z^n}^2 + \abs{w^n}^2}{2} \right) \\
&= \frac{1}{2}(\norm{\varphi(z^1 + jw^1)}, \dots, \norm{\varphi(z^n + j w^n)}).
\end{align*}
It follows that, for any $\br \in (\R_{>0})^n$, the application $\Phi$ maps $\momentummap_{T_{U(n)}}^{-1}(\frac{1}{2}\br)$ onto $\confspace{\br}$. 
Identifying $\twograss{n}$ with a (co)adjoint orbit, we obtain a canonical symplectic structure on $\twograss{n}$ an the above result rephrases as:
\begin{proposition}[Nohara, Ueda {\cite[Proposition 2.2]{nohara2014toric}}]
	\label{p:confspace_is_symplectic_reduction_of_grassmannians}
	The moduli space $\confspace{\br}$ of polygons with side length $\br$ in $\R^3 \isomorphic I\Hh$ is isomorphic to the symplectic reduction of $\twograss{n}$ by the $T_{U(n)}$-action at the value $\frac{1}{2}\br$.
\end{proposition}

\subsection{Completely integrable systems on $\twograss{n}$}

Let us recall here how Nohara and Ueda \cite{nohara2014toric} defined a family of completely integrable systems on $\twograss{n}$, one for each maximal family of disjoint diagonals in the planar convex regular polygon with $n$ sides, that generalizes systems of bending flows on $\confspace{\br}$.

Given a subset $I$ of $\lbrace 1, \dots, n \rbrace$, define a subgroup $U_I$ of $U(n)$ as the set of matrices $A = (a_{i,j})_{1 \leq i,j \leq n} \in U(n)$ such that
\[ (a_{i,j})_{i,j \in I} \in U(\card I) \quad \text{and} \quad a_{i,j} = \delta_{i,j} \text{ for } (i,j) \notin I \times I. \]

To a formal side $q^i$ of some polygon $\bq$ we associate the subgroup 
\[ U_{q^i} = U_{\lbrace i \rbrace} = \begin{pmatrix}
	I_{i-1} & 0 & 0 \\
	0 & U(1) & 0 \\
	0 & 0 & I_{n-i}
\end{pmatrix}, \]
where $I_k$ denotes the identity matrix of size $k \times k$. The momentum map $\psi_{q^i} : \twograss{n} \rightarrow \R$ of the action of $U_{q^i}$ on $\twograss{n}$ is defined by
\[ \psi_{q^i}(\class{\bz, \bw}) = \frac{\abs{z^i}^2 + \abs{w^i}^2}{2}. \]

More generally, to some diagonal $d = \sum_{i \in I} q^i$ we associate the subgroup $U_{d} = U_I$. Its momentum map $\momentummap_{U_d} : \twograss{n} \rightarrow \sqrt{-1} \mathfrak{u}(\card I)$ is given by 
\[ \momentummap_{U_d}(\class{\bz, \bw}) = \left( \frac{z^i \bar{z}^j + w^i \bar{w}^j}{2} \right)_{i,j \in I}. \]
The matrix $\momentummap_{U_d}(\class{\bz, \bw})$ has rank two and real eigenvalues, denote by 
\[ \lambda_{d,1}(\class{\bz, \bw}) \geq \lambda_{d,2}(\class{\bz, \bw}) \geq 0 \]
its first two eigenvalues. We will restrict our attention to the second one and define the second-eigenvalue fonction $\psi_d = \lambda_{d,2} : \twograss{n} \rightarrow \R$.

\begin{proposition}[Nohara, Ueda {\cite[Proposition 4.5]{nohara2014toric}}]
	Let $d_1, \dots, d_{n-3}$ be a maximal family of disjoint diagonals. Then 
	\[ \lbrace \psi_{q^1}, \dots, \psi_{q^n}, \lambda_{d_1,1}, \dots, \lambda_{d_{n-3},1}, \lambda_{d_1,2}, \dots, \lambda_{d_{n-3},2} \rbrace \]
	is a family of Poisson commutative functions on $\twograss{n}$.
\end{proposition}

Of course, this family of $3n - 6$ functions is too large to define a completely integrable system on the $(2n-4)$-dimensional manifold $\twograss{n}$. Actually, each adapted face of a polygon induces a linear dependence of some of these functions. Indeed, denote by $v_1, v_2, v_3$ the sides of an adapted face, where $v_i$ can be either a side $q^i$ or one of the chosen diagonals $d_\alpha$. There is a simple linear dependence between them, say $v_3 = v_1 + v_2$. It follows that $U_{v_1} \times U_{v_2}$ is a subgroup of $U_{v_3}$ and the respective momentum maps of these three groups satisfy:
\[ \momentummap_{U_{v_3}} = \begin{pmatrix}
	\momentummap_{U_{v_1}} & \ast \\
	\ast & \momentummap_{U_{v_2}}
\end{pmatrix} \]
Comparing the traces between these two matrices gives a linear relation in the above family.

However, getting rid of the redundant information we obtain a completely integrable system:
\begin{proposition}[Nohara, Ueda {\cite[Proposition 4.6]{nohara2014toric}}]
	\label{p:system_grassmaniann_induces_system_bending}
	The map
	\[ \Psi = (\Psi_d, \Psi_q) = (\psi_{d_1}, \dots, \psi_{d_{n-3}}, \psi_{q^1}, \dots, \psi_{q^{n-1}}) \]
	defines a completely integrable system on $\twograss{n}$. Its $n-3$ first components induce  via $\Phi : \twograss{n} \rightarrow \polspace{(2)}$ the systems of bending flows on $\confspace{\br}$ associated to the diagonals $d_1, \dots, d_{n-3}$ (up to sign and additive constant).
	
	More precisely, for any diagonal $d_\alpha = \sum_{i \in I_\alpha} q^i$, for any $\br \in (\R_{\geq 0})^n$ such that $\abs{\br} = 2$, and for any $\class{\bz, \bw} \in \Psi_q^{-1}(\frac{1}{2}r_1, \dots, \frac{1}{2}r_{n-1})$,
	\begin{equation}
		\label{eq:relation_Psi_and_F}
		4\psi_{d_{\alpha}}(\class{\bz, \bw}) = - f_\alpha \circ \Phi(\class{\bz, \bw}) + \sum_{i \in I_\alpha} r_i,
	\end{equation}
	where $f_\alpha(\class{\bq}) =\norm{ \sum_{i \in I_\alpha} q_i}$ maps a (class of) polygon $\class{\bq} \in \confspace{\br}$ to the length of its diagonal $d_\alpha$.
\end{proposition}

Let $\frac{1}{2} \brdot = \frac{1}{2} (r_1, \dots, r_{n-1}) \in \R^{n-1}$ be a value of $\Psi_q$. Suppose $\brdot$ satisfies
\begin{equation}
\label{eq:condition_r_proper}
r_1 > 0, \dots, r_{n-1} > 0 \text{ and } r_n := 2 - r_1 - \cdots - r_{n-1} > 0.
\end{equation}
Then $\Psi_q^{-1}(\frac{1}{2} \brdot)$ is exactly $\Phi^{-1}(\confspace{\br})$, the preimage by $\Phi$ of the moduli space of polygons with side lengths fixed to $\br = (r_1, \dots, r_n)$.

Fix now a value $\bc = (c_1, \dots, c_{n-3}) \in \R^{n-3}$ of $\Psi_d$. Then $\class{\bz, \bw}$ lies in $\Psi^{-1}(\bc, \frac{1}{2} \brdot) = \Psi_d^{-1}(\bc) \cap \Psi_q^{-1}(\frac{1}{2} \brdot)$ if and only if $\Phi(\class{\bz, \bw})$ lies in the fiber
\[ N = F^{-1}(c'_1, \dots, c'_{n-3}) \]
of the system of bending flows on $\confspace{\br}$, where each $c'_i$ is an affine transform (depending on $\br$) of $c_i$ that can be explicitly computed from Formula~\ref{eq:relation_Psi_and_F}. It follows that $\Psi^{-1}(\bc, \frac{1}{2} \brdot)$ is exactly $\Phi^{-1}(N)$, the preimage by $\Phi$ of this fiber.

Note that when $\brdot$ satisfies Condition~\ref{eq:condition_r_proper}, we have
\[ N \subset \confspace{\br} \subset \confspace{(2)}^\proper. \]
Hence each preimage above can be seen as a preimage by $\Phi^\proper$, for which we have the nice Proposition~\ref{p:Phi_is_principal_bundle}.

\subsection{Singular fibers of the systems on $\twograss{n}$}

Fix $n \geq 4$ and $d_1, \dots, d_{n-3}$ a choice of disjoint diagonals in an arbitrary planar convex $n$-gon. Consider the associated system $\Psi : \twograss{n} \rightarrow \R^{2n-4}$.

In this subsection, we give some facts that might suggest that the method we used in this paper to study the singular fibers of the system of bending flows on $\confspace{\br}$ could be applied as well to the system $\Psi$ on $\twograss{n}$.

\subsubsection*{Singular points} Let us follow the proof of \cite[Lemma 4.7]{nohara2014toric}. Remark that 
\[ \momentummap_{T_{U(n)}} = \frac{1}{2}(\psi_{q^1}, \dots, \psi_{q^n}), \]
hence the Hamiltonian vectors fields of the functions $\psi_{q^1}, \dots, \psi_{q^{n-1}}$ are linearly independent and span the $T_{U(n)}$-orbits. The maps $\psi_{d_1}, \dots, \psi_{d_{n-3}}$ corresponds to the unit bending vector fields $B_{d_1}, \dots, B_{d_{n-3}}$ under the identification provided by Proposition~\ref{p:confspace_is_symplectic_reduction_of_grassmannians}. It follows that if $\Phi(\class{\bz, \bw})$ is a regular point of the system on $\confspace{\br}$ (with $\br = (\psi_{q^1}(\class{\bz, \bw}), \dots, \psi_{q^n}(\class{\bz, \bw})$), then the Hamiltonian vector fields associated to $\psi_{d_1}, \dots, \psi_{d_{n-3}}$ at $\class{\bw, \bz}$ are linearly independent and transverse to the $T_{U(n)}$-fibers. It follows that $\class{\bz, \bw}$ is a regular value of the system $\Psi$ on $\twograss{n}$. This holds even for non generic $\br$ since we can work on the dense manifold of non lined polygons.

Conversely, if $\Phi(\class{\bz, \bw})$ is a singular point of the system on $\confspace{\br}$, then $\class{\bz, \bw}$ is a singular point of the system on $\twograss{n}$.

\subsubsection*{Lifting property}
In \S \ref{s:structure_of_singular_fibers}, we chose not to work with the maps $F$ and $f_1, \dots, f_{n-3}$ on $\confspace{\br}$, but rather with their lifts $\tilde{F}, \tilde{f_1}, \dots, \tilde{f}_{n-3}$ on $\polspace{\br}$. It is interesting to note that the same can be done with the system $\Psi$.
Namely, the functions $\psi_{q}, \lambda_{d, j} : \twograss{n} \rightarrow \R$ involved in the definition of $\Psi$ admit natural lifts
\[ \tilde{\psi}_{q}, \tilde{\lambda}_{d, j} : \twoframes{n} \rightarrow \R \]
with explicit expressions.

\subsubsection*{Decomposition into simpler fibers} An important step in \S \ref{s:structure_of_singular_fibers} is to notice that it suffices to work with prodigal fibers, because any non-prodigal fiber $\tilde{N}$ is isomorphic to a product $\tilde{N}_1 \times \cdots \tilde{N}_k$ of prodigal fibers of ``smaller'' systems. The same holds for a system on $\twograss{n}$.

Suppose that the value $(\bc, \brdot) \in (\R_{\geq 0})^{2n-4}$ is such that some $c'_\ell = 0$. That is to say, polygons $\tilde{\Phi}(\bz, \bw)$ satisfy (up to a cyclic permutation of the indices)
\begin{equation}
	\label{eq:polygon_Phi_closes_prematurely}
	\varphi(z^1 + j w^1) + \cdots + \varphi(z^k + j w^k) = 0
\end{equation}
for $k < n$ when $(\bz, \bw)$ lies in $\tilde{L} = \tilde{\Psi}^{-1}(\bc, \brdot)$. Set
\[ \begin{matrix}
	\bz_1 = (z^1, \dots, z^k), & \bz_2 = (z^{k+1}, \dots, z^n), \\
	\bw_1 = (z^1, \dots, z^k), & \bw_2 = (z^{k+1}, \dots, z^n).
\end{matrix} \]
It is immediate to check that Condition~\ref{eq:polygon_Phi_closes_prematurely} implies
\[ \alpha_1 (\bz_1, \bw_1) \in \twoframes{k} \quad \text{and} \quad \alpha_2 (\bz_2, \bw_2) \in \twoframes{n-k}, \]
where $\alpha_1, \alpha_2$ are two positive constants used to normalize: $\norm{\alpha_j \bz_j} = \norm{\alpha_j \bw_j} = 1$.
The function $\tilde{\psi}_{q^i}$ depends only on $z_i$ and $w_i$, so in particular it depends solely on either $(\bz_1, \bw_1)$ or $(\bz_2, \bw_2)$. Similarly, let $d = \sum_{i \in I} q_i$ be a diagonal. If $d$ is disjoint from the vanishing diagonal $d_\ell$, then either $I \subset I_1 = I_\alpha = \lbrace 1, \dots, k \rbrace$ or $I \subset I_2 = I_\alpha^\complement = \lbrace k+1, \dots, n \rbrace$ (up to replacing $I$ by its complement $I^\complement$, which geometrically doesn't change the diagonal $d$). It follows that $\tilde{f}$ depends only on $\lbrace q_i, i \in I_j \rbrace$ and the sum
\[ \sum_{i \in I_\alpha} r_i \]
can be expressed using only the components of $\br^j = (r_i)_{i \in I_j}$. By Formula~\ref{eq:relation_Psi_and_F}, $\tilde{\psi}_d$ then depends only on $(\bz_j, \bw_j)$.

The map $\tilde{\Psi}$ can then be split into two maps $\tilde{\Psi}_1 : \twoframes{k} \rightarrow \R^{2k - 4}$ and $\tilde{\Psi}_2 : \twoframes{n-k} \rightarrow \R^{2n - 2k - 4}$ such that $\tilde{\Psi}_j$ depends only on $(\bz_j, \bw_j)$, and
\[ \tilde{L} = \Psi^{-1}(\bc, \brdot) \isomorphic \Psi_1^{-1}(\bc^1, \brdot^1) \times \Psi_2^{-1}(\bc^2, \brdot^2) = \tilde{L_1} \times \tilde{L_2}. \]
Iterating the process, we can restrict the study to products of ``prodigal'' fibers and possible particular sets (typically $\twoframes{1}$ and $\twoframes{2}$, analogous to digons and triangles appearing in the case of polygons).

\medskip

A similar reduction might be possible when $\brdot$ has some component $r_\ell$ equal to zero. Indeed, the set of two frames $(\bz, \bw)$ in $\C^n$ satisfying $\psi_{q^\ell} = 0$ is naturally identified with the set of two frames in the hyperplane $\lbrace e_\ell = 0 \rbrace \isomorphic \C^{n-1}$. Formula~\ref{eq:relation_Psi_and_F} shows that suppressing $q_\ell$ and $r_\ell$ in the expression of any $\tilde{\psi}_d : \twoframes{n} \rightarrow \R$, one obtains the expression of some $\tilde{\psi}_{d'} : \twoframes{n - 1} \rightarrow \R$. The fiber $\tilde{L} = \Psi^{-1}(\bc, \br)$ can then be identified with the fiber $\tilde{L}'$ of some system $\tilde{\Psi}'$ on $\twoframes{n-1}$ obtained by removing $\tilde{\psi}_{q^\ell}$ and a redundant $\tilde{\psi}_d$.

\subsubsection*{Study of ``prodigal'' fibers} Suppose
\[ a,b,c \in \lbrace q^1, \dots, q^{n-1}, d_1, \dots, d_{n-3} \rbrace \]
are the sides of an adapted face $\Delta$ for the choice of diagonals $d_1, \dots, d_{n-3}$. Let $\tilde{L}$ be a singular fiber of $\Psi$ such that a nontrivial linear relation 
\begin{equation}
	\label{eq:degenerate_face}
	\alpha a + \beta b + \gamma c = 0
\end{equation}
holds in the polygon $\Phi(\class{\bz, \bw})$ when $\class{\bz, \bw} \in \tilde{L}$. Then we decompose a $2$-frame $(\bz, \bw)$ in $\tilde{L}$ into three smaller $2$-frames as we did for prodigal polygons in \S \ref{s:structure_of_singular_fibers}, as follows. To a side of $\Delta$, say $a$, we associate:
\[ (\bz_a, \bw_a) = \alpha_a \begin{pmatrix}
	\\
	(z^i, w^i)_{i \in I_a}\\
	\\
	\hline
	(z^a,  w^a)
\end{pmatrix}, \]
where $z^a, w^a$ in $\C^n$ and $\alpha_a > 0$ are (uniquely) chosen such that $(\bz_a, \bw_a)$ is a $2$-frame in $\C^{n_a}$, $n_a = \abs{I_a} + 1$. The non-crossing condition on the diagonals ensures that for each diagonal (or side) $\sum_I q^i$, the action of $U_I \subset U(n)$ on $\twoframes{n}$ induces naturally an action of $U'_I \subset U(n_a)$ on $\twoframes{n_a}$, with $U_I \isomorphic U'_I \isomorphic U(\abs{I})$. It also guarantees that the system $\tilde{\Psi}$ on $\twoframes{n}$ induces a system $\tilde{\Psi}_a$ on $\twoframes{n_a}$ such that the fiber $\tilde{L}$ is mapped onto a fiber $\tilde{L}_a$. More precisely, $\tilde{L}$ is isomorphic to a submanifold of $\tilde{L}_a \times \tilde{L}_b \times \tilde{L}_c$ characterized by Relation~\ref{eq:degenerate_face} (similarly to~\ref{eq:image_of_cutting}). The remaining question is then the existence of a result similar to Proposition~\ref{p:structure_prodigal_fiber}.

\subsubsection*{Isotropicness of the fibers}
Assuming the singular fibers of the system on $\twograss{n}$ are submanifolds, it is reasonable to expect that a vector tangent to a fiber can be approximated by vectors on neighboring fibers. More precisely, a vector $X$ tangent to a fiber $\tilde{N} = \tilde{\Psi}(\bc, \brdot)$ should be approximable by a sequence $(X_t)_{t>0}$ such that $X_t$ is tangent to a fiber $\tilde{N}_t = \tilde{\Psi}(\bc_t, \brdot_t)$, where $\bc_t \to \bc$ is chosen such that $N_t$ is ``less singular'' than $N$ (e.g. it provides polygons with a lower number of degenerate faces) and $\brdot_t \to \brdot$ is chosen such that it defines generic positive side lengths $\br_t$ for any $t > 0$. The isotropicness would follow by continuity, as in \S \ref{s:isotropicness_of_fibers}.

\subsection{Relation to Gel'fand--Cetlin}

Define the sequence of inclusions
$K_1 \subset \cdots \subset K_n = U(n)$
where $K_i$ is the group of matrices of the form
\[ \begin{pmatrix}
	A & 0 \\
	0 & T
\end{pmatrix} \]
with $A \in U(i)$, $T=\mathrm{diag}(\xi_1, \dots, \xi_{n-i})$, $\xi_1, \dots, \xi_{n-i} \in U(1)$. The dual of the Lie algebra $\mathfrak{k}_i$ of $K_i$ can be identified with the set of matrices of the form
\[ \begin{pmatrix}
X & 0 \\
0 & B
\end{pmatrix} \]
with $X$ an Hermitian $i \times i$ matrix, $B = \mathrm{diag}(\theta_1, \dots, \theta_{n-i})$, $\theta_1, \dots, \theta_{n-i} \in \R$. Under a similar identification, the coadjoint orbit of $U(n)$ through a Hermitian matrix $A$ is the set of all Hermitian matrices with same spectrum as $A$. In other words, a coadjoint orbit $\orbit{\blambda}$ is uniquely determined by a $n$-tuple $\blambda = (\lambda_1, \dots, \lambda_n) \in \R$ of fixed eigenvalues.

Given a matrix $M$ in some coadjoint orbit $\orbit{\blambda}$, denote by $M_k$ the upper-left submatrix of size $k \times k$ of $M$. The matrix $M_k$ as eigenvalues
\[ \mu_1^{k}(M) \geq \mu_2^{k}(M) \geq \cdots \geq \mu_k^{k}(M). \]
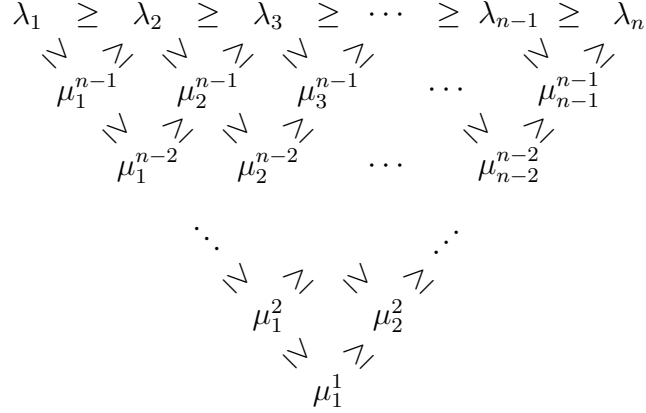
\begin{figure}
	\centering
	\begin{tikzpicture}[y=-1cm, x=0.8cm]
	\newcommand{\nodet}[2]{\node at (#1) {$#2$};}
	\newcommand{\nodel}[2]{\node at (#1, 0) {$\lambda_{#2}$};}
	\newcommand{\nodem}[3]{\node at (#1) {$\mu_{#3}^{#2}$};}
	\newcommand{\nodef}[1]{\node[rotate=-45] at (#1) {$\geq$};}
	\newcommand{\noder}[1]{\node[rotate=45] at (#1) {$\geq$};}
	
	\nodel{0}{1} \nodet{1,0}{\geq} \nodel{2}{2} \nodet{3,0}{\geq} \nodel{4}{3} \nodet{5,0}{\geq} \nodet{6,0}{\cdots}	\nodet{7,0}{\geq} \nodel{8}{n-1} \nodet{9,0}{\geq} \nodel{10}{n}
	
	\nodem{1,1}{n-1}{1} \nodem{3,1}{n-1}{2} \nodem{5,1}{n-1}{3} \nodet{7,1}{\cdots} \nodem{9,1}{n-1}{n-1}
	\nodef{0.5,0.5} \noder{1.5,0.5} \nodef{2.5,0.5} \noder{3.5,0.5} \nodef{4.5,0.5} \noder{5.5,0.5} \nodef{8.5,0.5} \noder{9.5,0.5} 
	\nodem{2,2}{n-2}{1} \nodem{4,2}{n-2}{2} \nodet{6,2}{\cdots} \nodem{8,2}{n-2}{n-2}
	\nodef{1.5,1.5} \noder{2.5,1.5} \nodef{3.5,1.5} \noder{4.5,1.5} \nodef{7.5,1.5} \noder{8.5,1.5}
	\node[rotate=-45] at (3, 3) {$\cdots$}; \node[rotate=45] at (7, 3) {$\cdots$};
	\nodef{3.5,3.5} \noder{4.5,3.5} \nodef{5.5,3.5} \noder{6.5,3.5}
	\nodem{4,4}{2}{1} \nodem{6,4}{2}{2}
	\nodef{4.5,4.5} \noder{5.5,4.5}
	\nodem{5,5}{1}{1}
	\end{tikzpicture}
	\caption{The Gel'fand--Cetlin diagram}
	\label{f:gelfand_cetlin_diagram}
\end{figure}
The Gel'fand--Cetlin system on $\orbit{\blambda}$ introduced by Guillemin and Sternberg~\cite{guillemin1983gelfand} is the one defined by the functions $\mu_{i}^k$, $1 \leq i \leq k \leq n$. It was orginally defined on generic coadjoint orbits, i.e. for $\blambda$ satisfying
\begin{equation}
	\label{eq:generic_gelfand_orbit_condition}
	\lambda_1 > \lambda_2 > \cdots > \lambda_n,
\end{equation}
but the definition can be extended to non-generic orbits as well. The functions of the Gel'fand--Cetlin system satisfy inequalities summarized in the Gel'fand--Cetlin diagram (Figure~\ref{f:gelfand_cetlin_diagram}). Regular points of this system are the matrices for which all the inequalities in the diagram are strict.

Back to the system on $\twograss{n}$, consider the caterpillar configuration where all the diagonals emanate from the same vertex, say the origin. That is, the family of disjoint diagonals
$\set{d_1, \dots, d_{n-3}}$
is defined by $d_\alpha = q^1 + \cdots + q^{\alpha+1}$. In this case we have a natural inclusion
\[ U_{q^1} \subset U_{d_1} \subset \cdots \subset U_{d_{n-3}} \subset U_{-q^n} \subset U(n) \]
where $U_{-q^n}$ denotes the subgroup associated to $q^1 + \cdots + q^{n-1}$. This induces a similar chain of subalgebras in $\mathfrak{u}(n)$. Proposition~9 of \cite{lane2015convexity} implies that the singular fibers of the system are connected, embedded submanifolds.

Let $M$ be the Hermitian matrix defined by
\[ M = \begin{pmatrix}
		\dfrac{z^i \bar{z}^j + w^i \bar{w}^i}{2}
\end{pmatrix}_{1 \leq i, j \leq n} \]
The upper-left submatrix $M_k$ of size $k$ of $M$ can be obtained as the product $M_k = \frac{1}{2} A_k A_k^\ast$ where
\[ A_k = \begin{pmatrix}
	z^1 & w^1 \\
	\vdots & \vdots \\
	z^k & w^k
\end{pmatrix} \]
is the matrix made of the $k$ first rows of $(\bz, \bw)$. Hence the nonzero eigenvalues of $M_k$ are the same as the nonzero eigenvalues of the $2 \times 2$ matrix $\frac{1}{2} A_k^\ast A_k$. Using this fact, one obtains
\[ \begin{cases}
	\mu_{1}^n(M) = \mu_2^n(M) = \frac{1}{2} & \\
	\mu_{1}^{n-1}(M) = \frac{1}{2} = \sum_{i=1}^{n-1}\tilde{\psi}_{q^i}(\bz, \bw) - \mu_2^{n-1}(M) & \\
	\mu_2^k(M) = \tilde{\psi}_{d_{k-1}}(\bz, \bw) = \sum_{i=1}^{k}\tilde{\psi}_{q^i}(\bz, \bw) - \mu_1^{k}(M) & \text{if } 2 \leq k \leq n-2 \\
	\mu_1^1(M) = \tilde{\psi}_{q^1}(\bz, \bw),\quad \mu_2^1(M) = 0
\end{cases} \]
and $\mu_i^k(M) = 0$ for $i > 2$.
\begin{figure}
	\centering
	\begin{tikzpicture}[y=-1cm, x=0.8cm]
		\newcommand{\nodet}[2]{\node at (#1) {$#2$};}
		\newcommand{\nodel}[2]{\node at (#1, 0) {$0$};}
		\newcommand{\nodem}[3]{\node at (#1) {$\mu_{#3}^{#2}$};}
		\newcommand{\nodef}[1]{\node[rotate=-45] at (#1) {$\geq$};}
		\newcommand{\noder}[1]{\node[rotate=45] at (#1) {$\geq$};}
		\newcommand{\nodefeq}[1]{\node[rotate=-45] at (#1) {$=$};}
		\newcommand{\nodereq}[1]{\node[rotate=45] at (#1) {$=$};}
		
		\nodet{0,0}{\frac{1}{2}} \nodet{1,0}{=} \nodet{2,0}{\frac{1}{2}} \nodet{3,0}{>} \nodel{4}{3} \nodet{5,0}{=} \nodet{6,0}{\cdots}	\nodet{7,0}{=} \nodel{8}{n-1} \nodet{9,0}{=} \nodel{10}{n}
		\nodefeq{0.5,0.5} \nodereq{1.5,0.5} \nodef{2.5,0.5} \noder{3.5,0.5}
		\nodet{1,1}{\frac{1}{2}} \nodem{3,1}{n-1}{2} \nodet{5,1}{0}
		\nodef{1.5,1.5} \noder{2.5,1.5} \nodef{3.5,1.5} \noder{4.5,1.5} 
		\nodem{2,2}{n-2}{1} \nodet{4,2}{\psi_{d_{n-3}}}
		\node[rotate=-45] at (3, 3) {$\cdots$}; \node[rotate=-45] at (5, 3) {$\cdots$}; \nodet{7, 3}{$0$};
		\nodef{3.5,3.5} \noder{4.5,3.5} \nodef{5.5,3.5} \noder{6.5,3.5}
		\nodem{4,4}{2}{1} \nodet{6,4}{\psi_{d_1}}
		\nodef{4.5,4.5} \noder{5.5,4.5}
		\nodet{5,5}{\psi_{q^1}}
		\end{tikzpicture}
	\caption{The Gel'fand--Cetlin diagram for the caterpillar configuration on $\twograss{n}$}
	\label{f:gelfand_cetlin_diagram_from_polygons}
\end{figure}
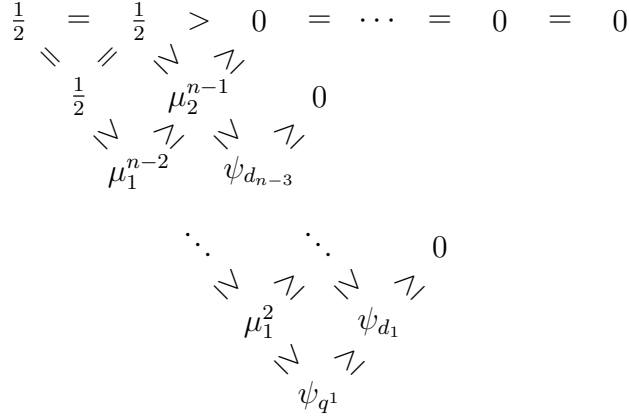
In other words, the Gel'fand--Cetlin system on the non-generic orbit $\orbit{\frac{1}{2}, \frac{1}{2}, 0, \dots, 0}$ is isomorphic to the system on $\twograss{n}$. In this case, under the identification $M=(\bz, \bw)$, the Gel'fand--Cetlin diagram becomes as in Figure~\ref{f:gelfand_cetlin_diagram_from_polygons} and the inequalities involved are exactly the triangle inequalities in the adapted faces for the caterpillar configuration, as already noticed by Hausmann and Knutson. More precisely:
\begin{theorem}[Hausmann, Knutson {\cite[Theorem 5.2]{HausKnut97}}]
	The bending flows for the caterpillar configuration on $\confspace{\br}$ are the residual torus action from the Gel'fand--Cetlin system on $\orbit{\frac{1}{2}, \frac{1}{2}, 0, \dots, 0}$.
\end{theorem}

\begin{figure}
	\centering
	\begin{tikzpicture}[y=-1cm, x=0.8cm]
		\draw[fill=lightgray, very thick] (0,0) -- (2,0) -- (1,1) -- cycle;
		\draw[fill=lightgray, very thick] (4,0) -- (12,0) -- (8,4) -- cycle;
		
		\draw[dashed] (1,1) -- (6,6) -- (7,5) -- (3,1) -- (2,2);
		\draw[dashed] (3,3) -- (4,2);
		\draw[dashed] (4,4) -- (5,3);
		\draw[dashed] (5,5) -- (6,4);
		
		\foreach \y in {0,...,2} {
			\pgfmathtruncatemacro{\xmax}{2-\y}
			\foreach \x in {0,...,\xmax}{
				\pgfmathtruncatemacro{\X}{6+2*\x+\y}
				\draw[very thick] (\X, \y) -- ++ (-1, 1) -- ++ (2, 0) -- cycle;
			}
		}
		
		\newcommand*{\nodec}[1]{\draw[fill] (#1) circle (2pt);}
		\foreach \y in {0,...,5} {
			\pgfmathtruncatemacro{\xmax}{12 - \y}
			\pgfmathtruncatemacro{\xstep}{\y + 2}
			\foreach \x in {\y,\xstep,...,\xmax}{
				\nodec{\x, \y}
			}
		}
		\nodec{6,6}
		
		\node at (3,2) {$D_{n-3}$};
		\node at (4,3) {$\ldots$};
		\node at (5,4) {$D_{2}$};
		\node at (6,5) {$D_{1}$};
	\end{tikzpicture}
	\caption{Diamonds in the graph $\Gamma_N$}
	\label{f:diamonds}
\end{figure}
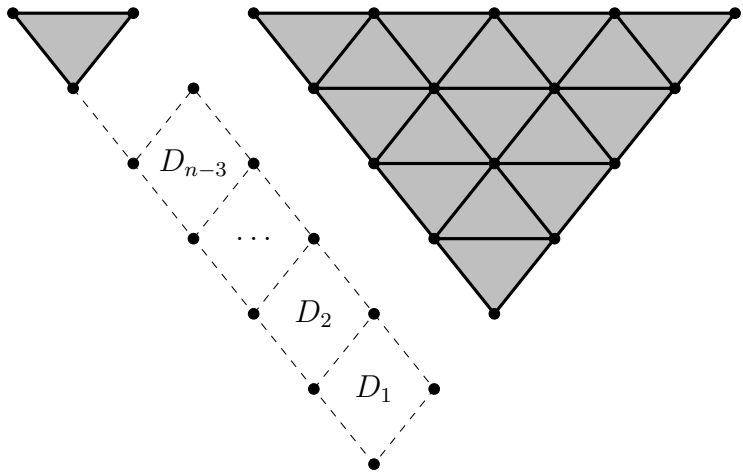

Fix positive side lengths $\br = (r_1, \dots, r_n) = (\brdot, r_n) \in (\R_{>0})^n$ and a value $\bc = (c_1, \dots, c_{n-3}) \in \R^{n-3}$. Let $N = f^{-1}(\bc)$ be the corresponding fiber in $\confspace{\br}$, and $L = \Psi^{-1}(\bc, \frac{1}{2}\brdot)$ the corresponding fiber in $\twograss{n}$. 
To the fiber $N$ is associated a graph as follows. The vertices are the functions appearing in the Gel'fand--Cetlin diagram, and there is an edge between two functions if and only if they are constant to the same value on $L$. This graph has the form illustrated in Figure~\ref{f:diamonds}, where a dashed edge correspond to the possible degeneracy of some adapted face of the polygons in $N$. The filled parts are common to all such graphs and can be ignored. Remark that
\[ 4 \mu_i^k = \sum_{j=1}^k r_j + (-1)^{i+1} c_{k-1} \]
along $L$, for any $i=1,2$ and $1 \leq k \leq n-1$ (with the convention $c_0 = r_1$ and $c_{n-2}=r_n$). The condition $\mu = \mu'$ becomes a condition of the form $\alpha_1 + \alpha_2 = \alpha_3$ with $\alpha_i \in \lbrace r_1, \dots, r_n, c_1, \dots, c_{n-3} \rbrace$ that can be explicitly checked (as mentioned above, this condition simply derives from a triangle inequality in an adapted face).

Thus from solely the combinatorics of the graph $\Gamma_N$ we can recover the geometric description of the fiber $N$. Of particular interest are the ``diamonds'' $D_1, \dots, D_{n-3}$. The existence of a diamond-shaped cycle $D_i$ in $\Gamma_N$ implies that the $i$-th diagonal $d_i$ vanishes on $N$, and in this case we know that (at least in the generic case) the fiber $N$ is geometrically the product of two spaces. The correspondence between the combinatorics of those diamonds and the geometry of the fibers was first established by Miranda and Zung \cite{mirandazung} for the classical Gel'fand--Cetlin system on $U(n)$. However in their case the diamond-shaped cycles can have bigger length and can cross each other, if they do then the geometry of the fiber is more subtle and involves cross-products.

\bibliographystyle{amsplain}
\bibliography{biblio}

\providecommand{\bysame}{\leavevmode\hbox to3em{\hrulefill}\thinspace}
\providecommand{\MR}{\relax\ifhmode\unskip\space\fi MR }
\providecommand{\MRhref}[2]{%
  \href{http://www.ams.org/mathscinet-getitem?mr=#1}{#2}
}
\providecommand{\href}[2]{#2}
\begin{thebibliography}{10}

\bibitem{alamiddine2009GelfandCeitlin}
I.~Alamiddine, \emph{G{\'e}om{\'e}trie de syst{\`e}mes {H}amiltoniens
  int{\'e}grables: le cas du syst{\`e}me de {G}elfand--{C}eitlin}, Ph.D.
  thesis, Universit{\'e} Toulouse III -- Paul Sabatier, 2009.

\bibitem{babelon2015higher}
O.~Babelon and B.~Doucot, \emph{Higher index focus--focus singularities in the
  {J}aynes--{C}ummings--{G}audin model: Symplectic invariants and monodromy},
  Journal of Geometry and Physics \textbf{87} (2015), 3--29.

\bibitem{bolsinov2014singularities}
A.~Bolsinov and A.~Izosimov, \emph{Singularities of bi-{H}amiltonian systems},
  Communications in Mathematical Physics \textbf{331} (2014), no.~2, 507--543.

\bibitem{bolsinov2004integrable}
A.~V. Bolsinov and A.~T. Fomenko, \emph{Integrable {H}amiltonian systems:
  geometry, topology, classification}, CRC Press, 2004.

\bibitem{cdt2014probability}
J.~Cantarella, T.~Deguchi, and C.~Shonkwiler, \emph{Probability theory of
  random polygons from the quaternionic viewpoint}, Communications on Pure and
  Applied Mathematics \textbf{67} (2014), no.~10, 1658--1699.

\bibitem{charles2010quantization}
L.~Charles, \emph{On the quantization of polygon spaces}, Asian J. Math.
  \textbf{14} (2010), no.~1, 109--152.

\bibitem{FOR09}
R.~L. Fernandes, J.-P. Ortega, and T.~S. Ratiu, \emph{The momentum map in
  {P}oisson geometry}, Amer. J. Math. \textbf{131} (2009), no.~5, 1261--1310.

\bibitem{guillemin1983gelfand}
V.~Guillemin and S.~Sternberg, \emph{The {G}elfand--{C}etlin system and
  quantization of the complex flag manifolds}, Journal of Functional Analysis
  \textbf{52} (1983), no.~1, 106--128.

\bibitem{harada2006symplectic}
M.~Harada, \emph{The symplectic geometry of the {G}el'fand--{C}etlin--{M}olev
  basis for representations of {$ Sp (2n,\mathbb{C}) $}}, Journal of Symplectic
  Geometry \textbf{4} (2006), no.~1, 1--41.

\bibitem{harada2012integrable}
M.~Harada and K.~Kaveh, \emph{Integrable systems, toric degenerations and
  {O}kounkov bodies}, Inventiones mathematicae (2012), 1--59.

\bibitem{HausKnut97}
J.-C. Hausmann and A.~Knutson, \emph{Polygon spaces and {G}rassmannians},
  Enseign. Math. (2) \textbf{43} (1997), no.~1-2, 173--198.

\bibitem{hmm2011toric}
B.~Howard, C.~Manon, and J.~Millson, \emph{The toric geometry of triangulated
  polygons in {E}uclidean space}, Canad. J. Math. \textbf{63} (2011), no.~4,
  878--937.

\bibitem{kamiyama2002symplectic}
Y.~Kamiyama and T.~Yoshida, \emph{Symplectic toric space associated to triangle
  inequalities}, Geometriae Dedicata \textbf{93} (2002), no.~1, 25--36.

\bibitem{KapMil96}
M.~Kapovich and J.~J. Millson, \emph{The symplectic geometry of polygons in
  {E}uclidean space}, J. Differential Geom. \textbf{44} (1996), no.~3,
  479--513.

\bibitem{lane2015convexity}
J.~Lane, \emph{Convexity and thimm's trick}, arXiv preprint arXiv:1509.07356
  (2015).

\bibitem{mirandazung}
E.~Miranda and N.~T. Zung, \emph{{G}elfand--{C}eitlin systems and their
  diamonds}, Private communication.

\bibitem{nohara2014toric}
Y.~Nohara and K.~Ueda, \emph{Toric degenerations of integrable systems on
  grassmannians and polygon spaces}, Nagoya Mathematical Journal \textbf{214}
  (2014), 125--168.

\bibitem{pelayo2014semiclassical}
A.~Pelayo and S.~Vu~Ngoc, \emph{Semiclassical inverse spectral theory for
  singularities of focus--focus type}, Communications in Mathematical Physics
  \textbf{329} (2014), no.~2, 809--820.

\bibitem{Pflaum01}
M.~J. Pflaum, \emph{Analytic and geometric study of stratified spaces}, Lecture
  Notes in Mathematics, vol. 1768, Springer-Verlag, Berlin, 2001.

\bibitem{Pfl03}
\bysame, \emph{On the deformation quantization of symplectic orbispaces},
  Differential Geom. Appl. \textbf{19} (2003), no.~3, 343--368.

\bibitem{zung1996symplectic}
N.~T. Zung, \emph{Symplectic topology of integrable {H}amiltonian systems, {I}:
  {A}rnold-{L}iouville with singularities}, Compositio Mathematica \textbf{101}
  (1996), 179--215.

\end{thebibliography}

\end{document}